\pdfoutput=1
\documentclass[a4paper,reqno]{amsart}

\usepackage[margin=3.5cm]{geometry}
\usepackage[T1]{fontenc}
\usepackage[utf8]{inputenc}
\usepackage[english]{babel}
\usepackage{amssymb}
\usepackage{mathtools}
\usepackage{amsthm}
\usepackage{enumitem}
\usepackage{caption}
\usepackage[dvipsnames]{xcolor}

\numberwithin{equation}{section}

\theoremstyle{definition}
\newtheorem{definizione}{Definition}[section]
\theoremstyle{plain}
\newtheorem{theorem}{Theorem}[section]

\newtheorem{lemma}[theorem]{Lemma}

\theoremstyle{definition}

\usepackage[numbers,sort&compress]{natbib}

\newcommand{\abs}[1]{{\left\vert #1 \right\vert}}
\newcommand{\norma}[1]{{\left\Vert #1 \right\Vert}}
\newcommand{\R}{\mathbb{R}}

\usepackage{hyperref}
\hypersetup{linktoc=none, bookmarksnumbered, colorlinks=true, linkcolor=magenta, citecolor=magenta}
\usepackage[nameinlink,noabbrev]{cleveref}
\crefformat{equation}{(#2#1#3)}
\crefrangeformat{equation}{(#3#1#4) to~(#5#2#6)}
\crefmultiformat{equation}{(#2#1#3)}{ and~(#2#1#3)}{, (#2#1#3)}{ and~(#2#1#3)}
\crefname{subsection}{section}{sections}
\Crefname{subsection}{Section}{Sections}
\renewcommand{\cref}[1]{\Cref{#1}}

\title{Finite Element Approximation of the Hardy constant}

\author[F. Della Pietra]{Francesco Della Pietra}
\address{
    Dipartimento di Matematica e Applicazioni ``R. Caccioppoli'', 
    Universit\`a degli studi di Napoli Federico II, 
    Via Cintia, Complesso Universitario Monte S. Angelo, 80126 Napoli, Italy.}
\email{f.dellapietra@unina.it}

\author[G. Fantuzzi]{Giovanni Fantuzzi}
\address{
    Friedrich-Alexander-Universität Erlangen-Nürnberg,
    Department of Mathematics, Chair for Dynamics, Control, Machine Learning and Numerics (Alexander von Humboldt Profes- sorship), 
    Cauerstr. 11, 91058 Erlangen, Germany.}
\email{giovanni.fantuzzi@fau.de}

\author[L. I. Ignat]{Liviu I. Ignat}
\address{
    Institute of Mathematics ``Simion Stoilow'' of the Romanian Academy, 21 Calea Grivitei Street, 010702 Bucharest, Romania.
    \newline\indent
    The Research Institute of the University of Bucharest - ICUB, University of Bucharest, 90-92 Sos. Panduri, 5th District, Bucharest, Romania
}
\email{liviu.ignat@gmail.com}

\author[A. L. Masiello]{Alba Lia Masiello}
\address{
    Dipartimento di Matematica e Applicazioni ``R. Caccioppoli'', 
    Universit\`a degli studi di Napoli Federico II, 
    Via Cintia, Complesso Universitario Monte S. Angelo, 80126 Napoli, Italy.
}
\email{albalia.masiello@unina.it}

\author[G. Paoli]{Gloria Paoli}
\address{
      Dipartimento di Matematica e Applicazioni ``R. Caccioppoli'', 
    Universit\`a degli studi di Napoli Federico II, 
    Via Cintia, Complesso Universitario Monte S. Angelo, 80126 Napoli, Italy}
\email{gloria.paoli@unina.it}

\author[E. Zuazua]{Enrique Zuazua}
\address{
    Friedrich-Alexander-Universität Erlangen-Nürnberg, 
    Department of Mathematics, Chair for Dynamics, Control, Machine Learning and Numerics (Alexander von Humboldt Professorship), 
    Cauerstr. 11, 91058 Erlangen, Germany.
    \newline\indent
    Chair of Computational Mathematics,
    Fundación Deusto, Avenida de las Universidades, 
    24, 48007 Bilbao, Basque Country, Spain.
    \newline\indent
    Universidad Autónoma de Madrid,
    Departamento de Matemáticas, 
    Ciudad Universitaria de Cantoblanco, 28049 Madrid, Spain.}
\email[Corresponding author]{enrique.zuazua@fau.de}

\date{}

\begin{document}
\maketitle

\begin{abstract} 
We consider finite element approximations to the optimal constant for the Hardy inequality with exponent $p=2$ in bounded domains of dimension $n=1$ or $n\geq 3$. For finite element spaces of piecewise linear and continuous functions on a mesh of size $h$, we prove that the approximate Hardy constant converges to the optimal Hardy constant at a rate proportional to $1/\abs{\log h}^2$. This result holds in dimension $n=1$, in any dimension $n\geq 3$ if the domain is the unit ball and the finite element discretization exploits the rotational symmetry of the problem, and in dimension $n=3$ for general finite element discretizations of the unit ball. In the first two cases, our estimates show excellent quantitative agreement with values of the discrete Hardy constant obtained computationally.


\end{abstract}

\section{Introduction}

In his celebrated work \cite{hardyvero}, G. H. Hardy proved that 
\begin{equation}\label{har}
  \left(\dfrac{p-1}{p}\right)^p\int_0^1 \dfrac{\abs{u}^p}{x^p} dx\le  \int_0^1 \abs{u'}^p \;dx
\end{equation}
for all $1<p<+\infty$ and all $u\in W^{1,p}(0,1)$ with $u(0)=0$. This inequality, nowadays called the \emph{Hardy inequality}, was extended in \cite{hardy} to open sets $\Omega \subseteq \R^n$ in $n\ge 2$ dimensions and $p \in (1,n)$, giving
\begin{equation}\label{hardy}
    \left(\frac{n-p}{p}\right)^p \int_{\Omega} \frac{\abs{u}^p}{\abs{x}^p} \;dx \le \int_{\Omega}\abs{\nabla u}^p \;dx
\end{equation}
for all $u\in W_0^{1,p}(\Omega)$. The inequality holds also when $\Omega=\R^n$ and it is trivial when $n=p$. 

The Hardy inequality has received considerable attention because it finds applications in several fields. For example, it is related to Heisenberg's \emph{uncertainty principle}~\cite{fefferman} and, for $p=2$, it is useful in describing properties of Schr\"odinger operators with inverse square potentials~\cite{frank}.
Further extensions of the inequality exist and the literature is broad. We refer readers to \cite{adimurthi,   cazacu, ddghardy, peral, zua_kre, prehistory,   kufnermal, kufnerpers,  soria} for a general overview.

It is well-known that the constants in \cref{har,hardy} are optimal, meaning that 
\begin{subequations}
    \begin{align}\label{e:hardy-ratio-def-1d}
     \left(\dfrac{p-1}{p}\right)^p 
     &=
     \inf_{\substack{u\in W^{1,p}(0,1)\\u(0)=0}} \; \frac{\displaystyle\int_0^1 \abs{u'}^p \;dx}{\displaystyle\int_0^1 x^{-p} \abs{u}^p \;dx}
     &&\text{if } n=1
    \intertext{and}
    \label{e:hardy-ratio-def-3d}
    \left(\dfrac{n-p}{p}\right)^p
    &=
     \inf_{u\in W^{1,p}_0(\Omega)}\frac{\displaystyle\int_\Omega \abs{\nabla u}^p \;dx}{\displaystyle\int_\Omega \abs{x}^{-p} \abs{u}^p \;dx}
     &&\text{if } n\geq 2.
\end{align}
\end{subequations}
These infima are not attained, but one can easily construct minimizing sequences. For example, one can approximate the function $u(x) = \abs{x}^{(p-n)/p}$ with functions in $W^{1,p}(\Omega)$ satisfying the correct boundary conditions.

In this work, we fix $p=2$ and consider the problem of approximating the optimal Hardy constant using the finite element method. 
Specifically, in dimension $n=1$, define the discrete Hardy constant as 
\begin{equation}\label{e:discrete-hardy-min}
    S_h=\min_{\substack{v\in V_h\\v(0)=0}} \; \frac{\displaystyle{\int_0^1 \vert v'\vert^2\,dx }}{\displaystyle{\int_0^1 x^{-2} \abs{v}^2\,dx}},
\end{equation}
where $V_h$ is the space of function in $H^{1}(0,1)$ that are piecewise linear on a `triangulation' of size $h$ (see \Cref{elementi} for a precise definition).
The approximation properties of $V_h$ in $H^1(0,1)$ guarantee that, as $h$ decreases, $S_h$ converges to the optimal value $1/4$ of the minimization problem in \cref{e:hardy-ratio-def-1d} for $p=2$. We prove that this convergence is logarithmic by establishing the following asymptotic expansion for $S_h$.

\begin{theorem}\label{main}
   For all sufficiently small triangulation size $h$,
    \begin{equation*}
        S_h =  \frac{1}{4} + \frac{\pi^2}{\abs{\log h}^2} + o\left(\dfrac{1}{\abs{\log h}^2} \right). 
    \end{equation*}
\end{theorem}

We also prove the same square logarithmic convergence, this time without the optimal prefactor, in $n\geq 3$ dimensions when the domain $\Omega=B$ is the unit ball. This restriction is justified because the Hardy constant is independent of the domain $\Omega$, and is convenient because then the minimization problem \cref{e:hardy-ratio-def-3d} enjoys a rotational symmetry. In particular, the minimization can be restricted to functions $u \in W^{1,p}_0(B)$ depending only on the radial coordinate $r$. Thus, the optimal Hardy constant for $p=2$ and dimension $n\geq 3$ is
\begin{equation}\label{S_n}
    S^n := \frac{(n-2)^2}{4} 
    = \inf_{\substack{u \in H^1(0,1) \\ u(1)=0}} \; \frac{\displaystyle{\int_0^1 r^{n-1}\abs{u'}^2\,dr }}{\displaystyle{\int_0^1 r^{n-3} \abs{u}^2\,dr}} 
    .
\end{equation}
We define its discrete version as
\begin{equation}\label{Snh}
    S^n_{h}=\min_{\substack{v\in V_{h} \\ v(1)=0} } \; \frac{ \displaystyle \int_0^1 r^{n-1} \abs{v'}^2 \, dr }{ \displaystyle \int_0^1 r^{n-3} \abs{v}^2\; dr}
\end{equation}
and prove the following statement.

\begin{theorem}\label{main-3d-spherical}
    For every $n\geq 3$ and every sufficiently small triangulation size $h$,
    \begin{equation*}
        S^n + \frac{\pi^2}{\abs{ \log h }^2}  +  o\left( \frac{1}{\abs{ \log h }^2} \right)
        \leq S^n_h \leq 
        S^n + \frac{(n+1)^2\pi^2}{4\abs{\log(h)}^2}
        + o\left(\frac{1}{\abs{\log h}^2}\right).
    \end{equation*}
\end{theorem}

Finally, in the special case of $n=3$ dimensions, we prove a square logarithmic convergence rate for the discrete Hardy constant even when the rotational symmetry of the unit ball is not exploited. Precisely, let $V_h^3$ be the space of functions in $H^{1}_0(B)$ that are piecewise linear on a general triangulation of the unit ball of $\mathbb{R}^3$ (see \Cref{elementi} for a precise definition) and recall from \cref{S_n} that $S^3 = 1/4$. We establish the following estimates. 

\begin{theorem}\label{main2}
Let $n = 3$ and let $V_h^3$ be a triangulation of $B$ of size $h$. There exists a positive constant $C$ such that, for every sufficiently small triangulation size $h$,
    \begin{equation*}
        \frac14 + \frac{C}{\abs{\log (h)}^2}\le \min_{v \in V_h^3 }\frac{\displaystyle\int_B \abs{\nabla v}^2 \;dx}{\displaystyle\int_B \abs{x}^{-2} \abs{v}^2 dx} 
       \leq \frac14 + 
       \frac{\pi^2}{\abs{\log h}^2}
       + o\left(\dfrac{1}{\abs{\log h}^2}\right).
    \end{equation*}
\end{theorem}

\noindent
The lower bound in this result holds in fact for any dimension $n\geq 3$, with the constant $\frac14$ replaced by the Hardy constant $S^n$ and with a constant $C$ that depends on $n$ (see  \cref{ss:lb-3d-liviu}).

Estimating the convergence rates for numerical approximations of optimal constants for functional inequalities is not a new problem. For example, approximations of the optimal Poincar\'e constant were studied in~\cite{buffa}, while convergence rates for finite element approximations to the Sobolev constant were established in~\cite{aldo}. There is also a related literature on estimating eigenvalues of operators, see for instance  \cite{Carstensen2014, Weinberger1956,guo_schr, sui}. 
While each of these problems presents its own challenges for numerical analysis, one can categorize functional inequalities into four broad classes with increasing complexity:
\begin{enumerate}[leftmargin=*, align=left]
    \item Inequalities where the equality is attained by a smooth function. This is the case, for example, for the Poincar\'e inequality in smooth domains.
    \item Inequalities where the equality is attained, but not by a smooth function. Examples in this class include Poincar\'e-type inequalities for elliptic operators in nonsmooth domains or with singular potentials.
    \item Inequalities where the equality is attained only when the underlying domain is the full space. The Sobolev inequality analyzed in \cite{aldo} belongs to this class.
    \item Inequalities where the equality is not attained, even on the full space.
\end{enumerate}

The Hardy inequality falls in the last class of problems and, as such, poses unique challenges. Indeed, to prove the upper bounds in \Cref{main,,main-3d-spherical,,main2} one can follow the strategy in \cite{aldo} and apply finite element interpolation estimates to minimizing sequences for the problems in \cref{e:hardy-ratio-def-1d,e:hardy-ratio-def-3d}. However, there are many possible minimizing sequences, so care must be taken to choose one with fast convergence properties. Finding lower bounds on the discrete Hardy constant is also not straightforward. In \cite{aldo}, the gap between the Sobolev constant and its finite element approximation was estimated from below using a quantitative version of the Sobolev inequality from \cite{fusco}, which estimates how far a function is from attaining equality.
Quantitative Hardy inequalities also exist (see, e.g., \cite{marcus, breva, vazu, Gazz} and \cite[Section 2.5]{peral}) and a version due to Wang \& Willem \cite{wang2003} suffices in dimension $n = 3$ to derive the lower bound in \cref{main2}.
For the lower bounds in \Cref{main,main-3d-spherical}, instead, we follow a strategy inspired by `calibration methods' from the calculus of variations (see, e.g., \cite[Section~1.2]{Buttazzo1998}), which is slightly more involved but is particularly well-suited to the one-dimensional nature of the variational problems in \cref{e:discrete-hardy-min} and \cref{Snh}.
The idea, loosely speaking, is to add to the Hardy inequality terms that integrate to zero and make the inequality evident.
This strategy is known to produce sharp estimates for principal eigenvalues of elliptic operators and of the $p$-Laplacian in dimension $n=1$ if $p$ is an even integer \cite{Chernyavsky2021}, and it has recently received attention in the optimization community because it lends itself to efficient numerical implementation \cite{Fantuzzi2022,Korda2018,Chernyavsky2021,Henrion2023}. Here, we use it to derive lower bounds for the discrete Hardy constant that not only show optimal dependence on the mesh size, but also exhibit an excellent quantitative agreement with computational results. The ability to produce explicit and accurate estimates is the main advantage of our `calibration' approach compared to using a quantitative Hardy inequality.

The rest of this article is organized as follows. 
\Cref{elementi} reviews basic notions of the finite element method. \Cref{main,main2} are proved in \Cref{s:1,s:3}, respectively. The proof of \Cref{main-3d-spherical}, instead, is relegated to \cref{s:proof-3d-radial} because the strategy is the same as for the one-dimensional case, but the computations are more cumbersome. \Cref{s:numerics} briefly compares the estimates in \cref{main,main-3d-spherical} to numerical values for the discrete Hardy constants obtained computationally for $n=1$ and $n=3$. \Cref{s:discussion} concludes the paper with a list of open problems.

\section{Finite Element Spaces}\label{elementi}

We start with a review of key notions about the finite element method. Readers are referred to \cite[Chapter 3]{quarteroni} and \cite{raviart} for details. We work in dimension $n=3$, but all results carry over to dimension $n = 1$ upon replacing polyhedra with intervals.

\begin{definizione}\label{approx}
    Let $\Omega \subset \mathbb{R}^n$ be a polyhedral domain (i.e., a finite union of polyhedra) and let $h>0$. A family $\mathcal{T}_h$ of polyhedra is called a \emph{triangulation} of $\overline{\Omega}$ if
    \begin{itemize}[leftmargin=*, align=left, noitemsep]
        \item Every $T\in \mathcal{T}_h$ is a subset of $\Omega$ with non-empty interior $\overset{\circ}{T}$;
        \item $ \overset{\circ}{T}_1\cap \overset{\circ}{T}_2=\emptyset $ for all $T_1\neq T_2 \in \mathcal{T}_h$;
        \item If $T_1\neq T_2 \in \mathcal{T}_h$ have $T_1\cap T_2\neq \emptyset$, then they share a common face, side or vertex;
        \item ${\rm diam} (T)\leq h$ for every $T\in \mathcal{T}_h$.
    \end{itemize}
    The vertices of the polyhedra in the triangulation $\mathcal{T}_h$ are called \emph{interpolation nodes}.
\end{definizione}

We restrict our attention to \emph{affine} triangulations, meaning that every element $T\in\mathcal{T}_h$ is the image of a reference polyhedron $\hat{T}$ under a $C^1$, invertible and affine map.
In particular, we will fix $\hat{T}$ to be the unit simplex.   
We also assume that the triangulations are \emph{shape regular}, meaning that there exists a constant $\sigma>0$ such that
\begin{equation*}
    \dfrac{h_T}{\rho_T}\leq \sigma \qquad \forall T\in \mathcal{T}_h,
\end{equation*}
where $\rho_T$ is the radius of the largest ball inscribed in $T$ and  $h_T$ is the diameter of $T$.
Finally, we impose that our meshes are \textit{uniform}, meaning  that  we require  $h/h_T$ to be uniformly bounded in $T\in\mathcal{T}_h$. As usual, for a given triangulation $\mathcal{T}_h$, we set without loss of generality
$$\displaystyle{h:=\max_{T\in\mathcal{T}_h}} \, h_T.$$ 

In dimension $n=3$, let $B$ be the open unit ball of $\mathbb{R}^3$. Let $B_h \subset B$ be an open polyhedral approximation of $B$ such that the boundary vertices of $\overline{B_h}$ lie on $\partial B$ and $\abs{B\setminus B_h}\leq h^2$. Such a polyhedral domain $B_h$ exists because $B$ is smooth and convex.
Let $\mathcal{T}_h$ be a triangulation of $B_h$ and denote by $V^3_h$ the space of functions in $H^1_0(B)$  that vanish on $B\setminus B_h$ and whose restriction to each element $T\in\mathcal{T}_h$ is linear. In dimension $n=1$, we define the space $V_h$ of continuous and piecewise linear functions on a triangulation (or \emph{mesh}) of $B_h=B=(0,1)$ in a similar way.

Next, we introduce  the finite element interpolation operator. 
\begin{definizione}\label{pro}
    The \emph{interpolation operator} $\Pi_h : C^0(\overline{B})\to V^3_h$ maps any continuous function $f$ to the continuous and piecewise linear function $\Pi_h f$ satisfying $\Pi_h f(x_i)=f(x_i),$ where $x_i$ are the interpolation nodes. 
\end{definizione}

In dimension $n\leq 3$ the interpolation operator is well-defined for every function in $H^2(B)$ because this space embeds continuously into $C^0(\overline{B})$. The following result is a restatement of \cite[Theorem 5.1-4]{raviart}.
\begin{theorem}\label{thint1}
   Let $\mathcal{T}_h$ be an affine, uniform and shape regular triangulation of a polyhedral domain $\Omega \subset \mathbb{R}^3$.
   There exists a constant $C_1>0$ such that, for every $f\in H^2(\Omega)$,
\begin{equation}\label{gradiente_pol}
    \norma{\nabla(\Pi_h f-f)}_{L^2(\Omega)}\le C_1 h \norma{D^2 f}_{L^2(\Omega)}.
\end{equation}
\end{theorem}

{A similar estimate holds if the polyhedral domain $\Omega$ is replaced by a $C^\infty$ domain (see, e.g., \cite[Lemma 5.2-3]{raviart}), except the $L^2$ norm of $D^2f$ must be replaced with the full $H^2$ norm of $f$.} For convenience, we recall this result only in the case of the ball $B$. 
\begin{lemma}\label[lemma]{thint2}
   Let $n=3$. There exists a constant $C_2>0$ such that, for every $f\in H^2(B)$,
    \begin{equation}\label{gradd}
    \norma{\nabla f}_{L^2(B\setminus B_h)}\le C_2 h \norma{f}_{H^2(B)}.
    \end{equation}
\end{lemma}

Combining \Cref{thint1}  and \Cref{thint2}, we obtain the following result. 
\begin{theorem}
    \label{thint}
Let $n=3$. There exists a constant $C>0$ such that, for every $f\in H^2(B)$,
    \begin{equation}\label{gradiente}
        \norma{\nabla(\Pi_h f-f)}_{L^2(B)}\le C h \norma{f}_{H^2(B)}.
    \end{equation}
\end{theorem}


\section{Proof of \texorpdfstring{\Cref{main}}{Theorem \ref{main}}}
\label{s:1}
This section is dedicated to proving  \Cref{main}. In \cref{ss:lb-1d}, we use a calibration-type argument to establish the lower bound 
\begin{equation}\label{e:lb-1d-statement}
    S_h \geq \frac14 + \left( \frac{\pi}{6+\abs{ \log h }} \right)^2 +o\left( \frac{1}{\abs{\log h}^2} \right).
\end{equation}
The argument, although technical, is interesting because it reveals a nontrivial good minimizing sequence for the minimization in \cref{e:hardy-ratio-def-1d}, which includes a sinusoidal term. We the interpolate a convenient approximation of this function in \cref{ss:lb-1d} to establish the upper bound 
\begin{equation}\label{e:ub-1d-statement}
    S_h \leq \frac14 + \left(\frac{\pi}{|\log h| - 3 \log |\log h|}\right)^2 +o\left(\dfrac{1}{\abs{\log h}^2} \right).
\end{equation}
This and \cref{e:lb-1d-statement} immediately imply the asymptotic expansion for $S_h$ stated in~\cref{main}.

Throughout this section, we shall assume for simplicity that the finite element space $V_h \subset H^1(0,1)$ is based on a uniform mesh whose elements have equal length $h$. All of our arguments, however, extend immediately to spaces $V_h$ defined using meshes with elements $[x_i,x_{i+1}]$ that satisfy $ch\leq x_{i+1}-x_i \leq h$ for some constant $c$ independent of $h$. Indeed, it suffices to replace $h$ with $ch$ in all of our proofs and results.

\subsection{Proof of the lower bound}
\label{ss:lb-1d}

Let $U_h$ be the space of functions in $H^1(0,1)$ that vanish at $x=0$ and are linear on $(0,h)$, but not necessarily on the rest of the interval $(0,1)$. Since the finite element space $V_h$ is strictly contained in $U_h$, we have that
\begin{equation}\label{e:mu-def}
    S_h > \mu_h := \inf_{u \in U_h} \dfrac{\displaystyle{\int_{0}^{1} \abs{u'}^{2}\,dx}}{\displaystyle{\int_{0}^{1}x^{-2}\abs{u}^{2}\,dx}}.
\end{equation}
This estimate, of course, is not expected to be sharp due to the strict gap between $V_h$ and $U_h$. However, as stated in the next theorem, we can compute $\mu_h$ exactly. This is enough to prove the lower bound in \cref{e:lb-1d-statement}.

\begin{theorem}\label{th:mu-h-bounds}
    There holds $\mu_h = 1/4 + \delta_h^2$, where $\delta_h$ solves
    \begin{equation}\label{e:eps-equation}
        \frac{1}{4}+\delta_h\tan\left(\tan^{-1}\frac{1}{2\delta_h} + \delta_h \log h \right)-\delta_h^{2} = 0.
    \end{equation}
    In particular, for $h\ll 1$ we have
    \begin{equation}\label{e:mu-asymptotic-1d}
         \mu_h = \frac14 + \left( \frac{\pi}{6+\abs{ \log h }} \right)^2 +  o\left( \frac{1}{\abs{ \log h }^2} \right).
    \end{equation}
\end{theorem}

This result could be established by solving the optimality conditions for the minimization problem defining $\mu_h$ in \cref{e:mu-def}. Here, however, we present an alternative strategy that applies in general and can produce estimates for $\mu_h$ from below even when the associated optimality conditions cannot be solved analytically. To ease the presentation we break the argument into three steps, which correspond to \cref{lem:1d-lower-bound,,lem:1d-asympt,lem:1d-lb-ub} below.
The first step is to prove the lower bound $\mu_h \geq 1/4 + \delta_h^2$ when $\delta_h$ solves \cref{e:eps-equation}.

\begin{lemma}[Lower bound on $\mu_h$]
\label[lemma]{lem:1d-lower-bound}
    Let $\delta_h$ satisfy \cref{e:eps-equation}. Then, $\mu_h \geq 1/4 + \delta_h^2$.
\end{lemma}
\begin{proof}
    For $\lambda \in \mathbb{R}$, set
    \begin{equation}\label{e:F-functional}
        F_\lambda(u) := \int_{0}^{1} \abs{u'}^{2} - \lambda \frac{u^2}{x^2}\,dx
    \end{equation}
    and observe that
    \begin{equation}\label{e:mu-ineq-form}
        \mu_h = \max\left\{ \lambda:\; F_\lambda(u) \geq 0 \quad \forall u \in U_h\right\}.
    \end{equation}
    Since every $u \in U_h$ has the linear representation $u(x) = (x/h) u(h)$ for $x\in(0,h)$, we can rewrite
    \begin{equation*}
        F_\lambda(u) = \frac1h \left( 1 - \lambda \right) u(h)^2 + \int_h^1  \abs{u'}^{2} - \lambda\, \frac{u^2}{x^2} \; dx.
    \end{equation*}
    We now use a calibration approach to find $\lambda$ for which $F_\lambda(u)$ is nonnegative irrespective of the choice of $u \in U_h$. Such a value $\lambda$ is then a lower bound on $\mu_h$.
    
    The idea is to add to $F_\lambda(u)$ terms that sum to zero and that, at least for some carefully chosen value of $\lambda$, make the inequality $F_\lambda(u)\geq 0$ evident. To this end, observe that if $\varphi$ is any continuously differentiable function on $[h,1]$ such that $\varphi(1)=0$, then the fundamental theorem of calculus gives
    \begin{equation*}
        \int_h^1 \left( \frac{\varphi(x)}{x} u^2 \right)' \, dx + \frac1h \varphi(h) u(h)^2 = 0.
    \end{equation*}
    After expanding the derivative inside the integral using the {product and chain rules}, we can add this expression to $F_\lambda(u)$ without changing its value to obtain
    \begin{equation}\label{e:F-calibrated}
        F_\lambda(u) = 
        \frac1h \left[ 1 - \lambda + \varphi(h) \right] u(h)^2 
        + \int_h^1  \abs{u'}^{2} + 2 \frac{\varphi}{x} u u' + \left(x \varphi' - \varphi - \lambda \right) \frac{u^2}{x^2}\;dx.
    \end{equation}
    The inequality $F_\lambda(u)\geq 0$ is satisfied if we can find $\varphi$ and $\lambda$ such that
    \begin{subequations}\label{e:phi-conditions}
         \begin{align}
            1 - \lambda + \varphi(h) &\geq 0, \label{e:phi-ineq-h}\\
            x \varphi' - \varphi - \lambda &\geq \varphi^2 \quad \forall x \in [h,1], \label{e:phi-ode}\\
            \varphi(1)=0. \label{e:phi-bc}
        \end{align}
    \end{subequations}
    Indeed, in this case we have
    \begin{equation}\label{e:F-nonneg}
        F_\lambda(u) 
        \geq 
        \frac1h \underbrace{\left[ 1 - \lambda + \varphi(h) \right]}_{\geq 0} u(h)^2 
        + \int_h^1 \underbrace{ \left( u' + \frac{\varphi}{x} u\right)^2 }_{\geq 0} \,dx,
    \end{equation}
    which is manifestly nonnegative for every function $u \in U_h$. 
    
    There remains to find $\varphi$ and $\lambda$ that satisfy the three conditions in~\cref{e:phi-conditions}. Fix $\delta > 0$ to be determined below and set
    $\lambda = \frac14 + \delta^2$.
    If we require that \cref{e:phi-ode} be satisfied with equality, we obtain a differential equation with the boundary condition $\varphi(1)=0$, whose solution is given by
    \begin{equation}\label{e:phi-choice}
        \varphi(x) = \delta \tan \left[ \tan^{-1}\left(\frac{1}{2\delta}\right) + \delta \log x \right] - \frac12.
    \end{equation}
    Note that this function is smooth on $[0,h]$ for $\delta$ small enough. Then, we substitute this function into \cref{e:phi-ineq-h} and rearrange the inequality to obtain
    \begin{equation}\label{e:eps-choice-condition}
        \frac{1}{4}+\delta\tan\left(\tan^{-1}\frac{1}{2\delta} + \delta \log h  \right)-\delta^{2}\geq0.
    \end{equation}
    This inequality holds with equality when $\delta = \delta_h$ is the solution of \cref{e:eps-equation}. All conditions in~\cref{e:phi-conditions} are then satisfied with equality. We conclude that $\lambda = 1/4 + \delta_h^2$ is feasible for the maximization problem in \cref{e:mu-ineq-form}, whence $\mu_h \geq \lambda = 1/4 + \delta_h^2$.
\end{proof}

The second step is to derive an asymptotic expansion for $\delta_h$ when $h\ll 1$.

\begin{lemma}[Asymptotic expansion for $\delta_h$]
\label[lemma]{lem:1d-asympt}
    Let $\delta_h$ solve \cref{e:eps-equation}. For $h\ll 1$, we have the asymptotic expansion
    \begin{equation}\label{e:delta-asymptotic}
        \delta_h = \frac{\pi}{6+\abs{ \log h }} + o\left( \frac{1}{\abs{ \log h }}\right).
    \end{equation}
\end{lemma}
\begin{proof}
    Rewrite \cref{e:eps-equation} as
    \begin{equation*}
        \tan^{-1}\left(\frac{1}{2\delta_h}\right) - \delta_h \abs{ \log h } = \tan^{-1} \left( \frac{4\delta_h^2 - 1}{4\delta_h}\right),
    \end{equation*}
    where we have used the identity $\log h = - \abs{\log h}$ valid for $h\leq 1$. Anticipating that $\delta_h \ll 1$ when $h\ll 1$, we can apply a Taylor expansion to find that
    \begin{equation*}
        \frac{\pi}{2} + {2\delta_h} + o(\delta_h) - \delta_h \abs{\log h} = \frac{\pi}{2} + 4\delta_h + o(\delta_h).
    \end{equation*}
    Solving for $\delta_h$ gives \cref{e:delta-asymptotic}. The correctness of this expansion is confirmed by \Cref{fig:asympt-epsilon}.
\end{proof}

\begin{figure}
    \centering
    \includegraphics[width=0.49\linewidth]{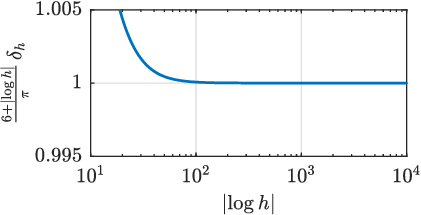}
    \hfill
    \includegraphics[width=0.49\linewidth]{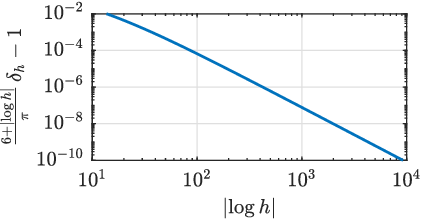}
    \caption{\emph{Left:} Ratio of $\delta_h$ to the leading-order term $\pi/(6+\abs{ \log h })$ in its asymptotic expansion. \emph{Right:} The error between this ratio and the value $1$. 
    }
    \label{fig:asympt-epsilon}
\end{figure}

The third and final step to prove \cref{th:mu-h-bounds} is to complement the lower bound on $\mu_h$ from \cref{lem:1d-lower-bound} with a matching upper bound.

\begin{lemma}\label[lemma]{lem:1d-lb-ub}
    Let $\delta_h$ satisfy \cref{e:eps-equation}. Then, $\mu_h \leq 1/4 + \delta_h^2$.
\end{lemma}
\begin{proof}
    It suffices to find a function $u_h \in U_h$ such that
    \begin{equation}
        \dfrac{\displaystyle{\int_{0}^{1} \abs{u_h'}^{2}\,dx}}{\displaystyle{\int_{0}^{1}x^{-2}\abs{u_h}^{2}\,dx}} = \frac14 + \delta_h^2.
        \label{e:lb-sharpness-ub}
    \end{equation}
    Setting $\lambda= 1/4 + \delta_h^2$ to ease the notation, this is equivalent to solving the equation $F_\lambda(u_h)=0$ where the functional $F_\lambda$ is as in \cref{e:F-functional}. Since the value of $\lambda$ was chosen to satisfy the conditions in \cref{e:phi-conditions} with equality, we find from \cref{e:F-calibrated} that
    \begin{equation}\label{e:1d-square-completion-b}
        F_{\lambda}(u_h) = \int_h^1 \left( u_h' + \frac{\varphi}{x} u_h\right)^2 \, dx
    \end{equation}
    for any $u_h \in U_h$.
    We should therefore take $u_h$ to solve the differential equation
    \begin{equation}\label{e:ode-u_h}
        u_h' + \frac{\varphi}{x} u_h = 0
    \end{equation}
    on $(h,1)$, and extend it by a linear function to $(0,h)$ while ensuring that $u_h(0)=0$. The differential equation \cref{e:ode-u_h} can be solved analytically if $\varphi$ is as in \cref{e:phi-choice} with $\delta=\delta_h$, giving
    \begin{equation}\label{e:near-opt-u-1d}
        u_h(x) =\begin{dcases}
         A \, \frac{x}{\sqrt{h}} \cos\!\left( \tan^{-1}\left(\frac{1}{2\delta_h}\right) + \delta_h  \log h \right)
         &\text{for } x \in [0,h]
         \\[1ex]
        A \, \sqrt{x} \; \cos\!\left( \tan^{-1}\left(\frac{1}{2\delta_h}\right) + \delta_h  \log x \right)
         &\text{for } x \in (h,1]
        \end{dcases}
    \end{equation}
    for an arbitrary normalization constant $A$. This function satisfies \cref{e:lb-sharpness-ub} by construction for any $A\neq 0$, which is the desired result.
\end{proof}

We conclude by remarking that \cref{lem:1d-lb-ub} is not required to obtain the lower bound on $S_h$ stated in \cref{e:lb-1d-statement}: that result already follows from \cref{lem:1d-lower-bound}, \cref{lem:1d-asympt}, and inequality \cref{e:mu-def}. Nevertheless, the extra analysis is valuable because it provides functions $u_h \in U_h$ that, as $h \to 0$, form a good minimizing sequence for the minimization problem defining the Hardy constant in \cref{e:hardy-ratio-def-1d}. In the next section, we interpolate an approximation of $u_h$ to estimate the discrete Hardy constant $S_h$ from above with optimal errors.

\subsection{Proof of the upper bound}
\label{ss:ub-1d}

We now prove the upper bound in \cref{e:ub-1d-statement}. Recall that, in dimension $n=1$, the discrete Hardy constant $S_h$ is the optimal value of the optimization problem in \cref{e:discrete-hardy-min}. It therefore suffices to construct a function $v_h \in V_h$ such that
\begin{equation}\label{e:ub-1d-strategy-2}
    \frac{\displaystyle{\int_0^1 \abs{v_h'}^2 \, dx}}{\displaystyle{\int_0^1 x^{-2}\, {v_h^2}\; dx}} 
    \leq 
    \frac14 + 
    \left(\frac{\pi}{|\log h| - 3 \log |\log h|}\right)^2 +o\left(\dfrac{1}{\abs{\log h}^2} \right).
\end{equation}

We will take $v_h$ to be the piecewise linear interpolation of an element $v_\varepsilon$ of a minimizing sequence $\{v_\varepsilon\}_{\varepsilon > 0}$ for the minimization problem in \cref{e:hardy-ratio-def-1d}. The construction requires a suitable choice of $\varepsilon$ as a function of the mesh size $h$ and, most importantly, a good choice of $v_\varepsilon$. Indeed, there are many possible minimizing sequences $\{v_\varepsilon\}_{\varepsilon > 0}$ for~\cref{e:hardy-ratio-def-1d}, and not all converge at the same rate as $\varepsilon$ tends to zero. The lower bound analysis of \cref{ss:lb-1d} suggests one should define $v_\varepsilon$ by replacing $h$ with $\varepsilon$ in \cref{e:near-opt-u-1d}. To simplify the algebra in what follows, however, it will be more convenient to work with the function
\begin{equation}\label{min1}
    v_\varepsilon(x)=
    \begin{cases}
    0,
    & x\in (0,\varepsilon),\\
    \sqrt{x}\sin\left(\dfrac{\pi \log x}{\log \varepsilon}\right), 
    & x\in (\varepsilon,1),
    \end{cases}
\end{equation}
which approximates the function in \cref{e:near-opt-u-1d} for small $\varepsilon$. Crucially, this function is linear on the interval $(0,\varepsilon)$. If we choose $\varepsilon = mh$ to be an interpolation node, therefore, $v_\varepsilon$ coincides with its piecewise linear interpolation $\Pi_h v_\varepsilon$ on $(0,\varepsilon)$. On the interval $(\varepsilon,1)$, instead, we can estimate the error between $v_\varepsilon$ and $\Pi_h v_\varepsilon$ using \cref{thint1} because $v_\varepsilon \in H^2(\varepsilon,1)$. This allows us to establish \cref{e:ub-1d-strategy-2} for $v_h = \Pi_h v_\varepsilon$ and a suitable choice of $\varepsilon$.

We start by calculating the values of some norms of $v_\varepsilon$.

\begin{lemma}\label[lemma]{lem:norm-v-eps}
    For every $\varepsilon<1$, the function $v_\varepsilon$ in \cref{min1} satisfies
    \begin{gather*}
        \int_0^1 \frac{v_\varepsilon(x)^2}{x^2} \, dx 
        = \frac12 \abs{\log \varepsilon},\\
        \int_0^1 \abs{v_\varepsilon(x)'}^2 \, dx 
        = \frac12 \abs{\log \varepsilon} \left( \frac14 + \frac{\pi^2}{\abs{\log \varepsilon}^2}\right), \\
        \int_\varepsilon^1 \abs{v_\varepsilon(x)''}^2 x^2 \, dx 
        = \frac{1}{32}\abs{\log \varepsilon} + \frac{\pi^2}{4 \abs{\log \varepsilon}} + \frac{\pi^4}{2\abs{\log \varepsilon}^3}.
    \end{gather*}
\end{lemma}
\begin{proof}
    By direct calculation.
\end{proof}

We will also use the following estimates, which relate a function $f\in H^1(0,1)\cap H^2(\varepsilon,1)$ that vanishes on $[0,\varepsilon]$ to its piecewise linear interpolation $\Pi_h f$ on a mesh of size $h$ when $\varepsilon=mh$ is an interpolation node. The proof is analogous to that of \cref{lem:estimates} in the appendix, so we do not report it for brevity.

\begin{lemma}\label[lemma]{lem:1d-estimates}
    Let $\varepsilon=mh\in(0,1)$ be an interpolation node. Assume $f\in H^1(0,1)\cap H^2(\varepsilon,1)$ vanishes on $[0, \varepsilon]$. Set
    \begin{equation*}
    \mathcal{E}_h(f):=\frac{h}{\varepsilon}
    \left(\int_0^1 \abs{f'}^2 \, dx\right)^{\frac{1}{2}}    \left(\int_\varepsilon^1 \abs{f''}^2 x^2 \, dx\right)^{\frac{1}{2}}+\frac{h^2}{\varepsilon^2}\int_\varepsilon^1\abs{f''}^2 x^2 \,dx.
    \end{equation*}
    There exists a constant $C > 0$, independent of $f$, $h$ and $\varepsilon$, such that
    \begin{subequations}
    \begin{align}
    \label{e:1d-first_e}
        \int_0^1 \abs{\left(\Pi_h f\right)'}^2 \, dx &\leq \int_0^1\abs{f'}^2\, dx+C \mathcal{E}_h(f)
        \\
        \label{e:1d-sec_e}
        \int_0^1 x^{-2} \abs{\Pi_h f}^2 \, dx &\geq \int_0^1x^{-2} \abs{f}^2 \, dx - C \mathcal{E}_h(f).
    \end{align}
    \end{subequations}
\end{lemma}

We are now ready to prove that the piecewise linear function $v_h = \Pi_h v_\varepsilon$ satisfies \cref{e:ub-1d-strategy-2} when $v_\varepsilon$ is as in \cref{min1} and $\varepsilon$ is a carefully chosen interpolation node. Precisely, set $\varepsilon = mh$ for some integer $m$ to be specified below and observe that $v_\varepsilon(x) = 0$ for all $x \in [0, \varepsilon]$. Then, we can apply \cref{e:1d-first_e,e:1d-sec_e} to estimate
\begin{equation}
\label{e:proof-ub-step-1}
    \frac{\displaystyle{\int_0^1 \abs{(\Pi_h v_\varepsilon)'}^2 \, dx}}{\displaystyle{\int_0^1 x^{-2}\, {|\Pi_h v_\varepsilon|^2}\; dx}} \leq
    \frac{\displaystyle{\int_0^1 \abs{v_\varepsilon'}^2 \, dx + C \mathcal{E}_h(v_\varepsilon)}}{\displaystyle{\int_0^1 x^{-2}\, {|v_\varepsilon|^2}\; dx} - C \mathcal{E}_h(v_\varepsilon)}
\end{equation}
Using the calculations reported in \cref{lem:norm-v-eps} we find that
$\mathcal{E}_h(v_\varepsilon)\lesssim 
    \left( h \varepsilon^{-1}+h^2 \varepsilon^{-2} \right) \abs{\log \varepsilon},$ so there exist a constant $C$, different from the one in \cref{e:proof-ub-step-1} but still independent of $\varepsilon$ and $h$, such that
\begin{equation}\label{e:ub-step-2}
    \frac{\displaystyle{\int_0^1 \abs{(\Pi_h v_\varepsilon)'}^2 \, dx}}{\displaystyle{\int_0^1 x^{-2}\, {|\Pi_h v_\varepsilon|^2}\; dx}}
    \leq \frac{\dfrac14 + \dfrac{\pi^2}{\abs{\log \varepsilon}^2} + C\left(\dfrac{h^2}{\varepsilon^2}+\dfrac{h}{\varepsilon}\right)}{1 - C \left(\dfrac{h^2}{\varepsilon^2}+\dfrac{h}{\varepsilon} \right)}.
\end{equation}
Next, set $m=\lfloor |\log h|^3 \rfloor$, so $\varepsilon=mh \sim h \abs{\log h}^3$ satisfies in particular $\varepsilon \leq h \abs{\log h}^3$. With this choice we can estimate 
$$\abs{\log \varepsilon}^2 \geq  \abs{ \log\left( h \abs{\log h}^3 \right) }^2 = (|\log h| - 3 \log |\log h|)^2$$
and
\begin{equation*}
    \dfrac{h^2}{\varepsilon^2}+\dfrac{h}{\varepsilon}
    \sim \frac{1}{\abs{\log h}^6} + \frac{1}{\abs{\log h}^3}
    = o\left( \frac{1}{\abs{\log h}^2}\right).
\end{equation*}
If we substitute these estimates into \cref{e:ub-step-2} and take $h\ll 1$, so we can apply the inequality $1/(1-z) \leq 1+2z$ valid for $z\leq 1/2$, we obtain
\begin{align*}
    \frac{\displaystyle{\int_0^1 \abs{(\Pi_h v_\varepsilon)'}^2 \, dx}}{\displaystyle{\int_0^1 x^{-2}\, {|\Pi_h v_\varepsilon|^2}\; dx}}
    &\leq 
    \left[ \dfrac14 + \left(\frac{\pi}{|\log h| - 3 \log |\log h|}\right)^2 +o\left( \frac{1}{|\log h|^2}\right) \right] \left[ 1 + o\left( \frac{1}{|\log h|^2}\right)\right]
    \\
    &=
    \dfrac14 + \left(\frac{\pi}{|\log h| - 3 \log |\log h|}\right)^2  +o\left( \frac{1}{|\log h|^2} \right).
\end{align*}
This is precisely \cref{e:ub-1d-strategy-2} for $v_h = \Pi_h v_\varepsilon$, which implies the upper bound on $S_h$ claimed in~\cref{e:ub-1d-statement}.

We conclude by remarking that while the choice of $m$ could in principle be optimized, this can only improve the terms that are asymptotically smaller than $1/\abs{\log h}^2$ when $h \ll 1$. The leading-order term $\pi^2 / \abs{\log h}^2$, instead, is optimal in light of the lower bound in \cref{e:lb-1d-statement}.

\section{Proof of \texorpdfstring{\Cref{main2}}{Theorem \ref{main2}}}\label{s:3}

We now turn to proving the upper and lower bounds from \cref{main2}, which apply to the discrete Hardy constant in $n=3$ dimensions for general triangulations of the unit ball. The main difference with the arguments for $n=1$ dimensions is in the proof of the lower bound, which we present in \cref{ss:lb-3d-liviu}: rather than following a calibration argument, we exploit a Hardy inequality with a logarithmic remainder term \cite{wang2003}. We note, however, that the proof of this inequality given in \cite[Section~2.5]{soria} relies on a completion-of-the-square argument, very similar in spirit to what our calibration strategy achieves in \cref{e:F-nonneg,e:1d-square-completion-b}. We remark also that while in \cref{ss:lb-3d-liviu} we fix $n=3$, our arguments immediately generalize to any dimension $n\geq 3$ and yield
\begin{equation*}
    \min_{v \in V_h^n}
    \frac{\displaystyle\int_B \abs{\nabla v}^2 \;dx}{\displaystyle\int_B \abs{x}^{-2} \abs{v}^2 dx}
    \geq S^n + \frac{C_n}{\abs{\log (h)}^2},
\end{equation*}
where $V_h^n$ is the space of piecewise linear functions on a triangulation of the $n$-dimensional unit ball, $S^n=(n-2)^2/4$ is the Hardy constant, and the constant $C_n>0$ depends on $n$.

The upper bound part of \cref{main2}, instead, is proved in \cref{ss:ub-3d-general} with the same interpolation strategy used for $n=1$. For this we must fix $n=3$ because the finite element interpolation estimates from \cref{thint} are not valid in higher dimensions.

\subsection{Proof of the lower bound}
\label{ss:lb-3d-liviu}
We exploit the following Hardy inequality with a logarithmic remainder term. The statement, adapted from \cite[Theorem~2.5.2, p.~25]{soria}, is a particular case of general quantitative Caffarelli--Kohn--Nirenberg inequalities from \cite{wang2003}.

\begin{theorem}[See {\cite[Theorem~2.5.2]{soria}}]
    Let $R>0$ be such that $0\in \Omega\subset \overline{\Omega}\subset B_R(0)$. There exists a positive constant $K(n,R)$ such that every $\phi\in C_c^\infty(\Omega)$ satisfies
    \begin{equation}
        \label{hardy.log}
        \int_\Omega |\nabla \phi|^2 \,dx- \left(\frac{n-2}2\right)^2 \int_\Omega \abs{x}^{-2} \phi^2 \,dx
        \geq K \int_\Omega |\nabla \phi|^2\left(\log \frac{|x|}{R} \right)^{-2} \,dx.
    \end{equation}
\end{theorem}

By density, this inequality holds for all $\phi\in H_0^1(\Omega)$. Because of the definition of space $V_h^3$ in  \Cref{elementi}, we can take $R=2$ and apply the above inequality to any function in $V_h^3$.

In particular, let $v_h\in V_h^3$ be the function attaining the minimum in the definition of $S^3_h$, that is,
\[
S^3_h=\dfrac{\displaystyle{\int_{B} |\nabla v_h|^2 \, dx}}{\displaystyle{\int_{B} \abs{x}^{-2} \, v_h^2\; dx}}.
\]
Using inequality \cref{hardy.log} for $n=3$ we obtain
\begin{align*}
S^3_h - \frac14
&=\dfrac{\displaystyle{\int_{B} |\nabla v_h|^2 \, dx}- \frac14\int_{B} \abs{x}^{-2} \, v_h^2\; dx}{\displaystyle{\int_{B} \abs{x}^{-2} \, v_h^2\; dx}}
\\
&\ge K \; \dfrac{\displaystyle{\int_{B}  {|\nabla v_h|^2}{\left(\log \tfrac{|x|}{2} \right)^{-2} \, dx}}}{\displaystyle{\int_{B} \abs{x}^{-2} \, v_h^2\; dx}}
= K \, \dfrac{\displaystyle{\sum _{T\in \mathcal{T}_h}\int_{T}  {|\nabla v_h|^2}{\left(\log \tfrac{|x|}{2} \right)^{-2}} \, dx}}{\displaystyle{\int_{B} \abs{x}^{-2} \, v_h^2\; dx}} \; .
\end{align*}
We claim that there exists a constant $C_1>0$ such that
 \begin{align}
         \sum _{T\in \mathcal{T}_h}\int_{T}  |\nabla v_h|^2 \left(\log \tfrac{|x|}{2} \right)^{-2} \, dx
         &\ge   \frac{C_1}{|\log h|^2}\sum _{T\in \mathcal{T}_h}\int_{T}  |\nabla v_h|^2\, dx
         \nonumber \\
         &=\frac{C_1}{|\log h|^2}\int_B |\nabla v_h|^2\, dx.
         \label{claim-liviu}
\end{align}
Then, upon setting $C=K C_1$ and using \cref{e:hardy-ratio-def-3d} for $n=3$ and $p=2$, we obtain
\begin{equation*}
S^3_h-\frac14
\ge  \frac C{|\log h|^2}\dfrac{\displaystyle{\int_{B} |\nabla v_h|^2 \, dx}}{\displaystyle{\int_{B} \abs{x}^{-2} \, v_h^2\; dx}}
\geq \frac{C}{4|\log h|^2},
\end{equation*}
which immediately implies the lower bound in \Cref{main2}.

There remains to prove \cref{claim-liviu}. 
Let $B_h$ be the ball of radius $h$ centered at the origin and observe that we can write $\mathcal{T}_h= \mathcal{T}_h^1 \cup \mathcal{T}_h^2$, where
\begin{align*}
    \mathcal{T}_h^1 &:= \{T\in \mathcal{T}_h \text{ such that } T\cap B_h= \emptyset\}\\
    \mathcal{T}_h^2 &:= \{T\in \mathcal{T}_h \text{ such that } T\cap B_h\neq \emptyset\}.
\end{align*}
For any $T\in \mathcal{T}_h^1$,  $x\in T$ implies $|x|\ge h$, so
\begin{equation}
    \label{es:th1}
    \int_{T}\frac{|\nabla v_h|^2}{|\log (|x|/2)|^2}\, dx\ge \frac{1}{\abs{\log (h/2)}^2} \int_T |\nabla v_h|^2\,dx \qquad \forall T \in \mathcal{T}_h^1.
\end{equation}
On the other hand, if $T\in \mathcal{T}_h^2$, we can write
\begin{equation*}
    \begin{aligned}
        \sum _{T\in \mathcal{T}_h^2}\int_{T}  \frac{|\nabla v_h|^2}{|\log (|x|/2)|^2} \, dx
        =& 
        \sum _{T\in \mathcal{T}_h^2} \left[
        \int_{T\cap B_h}  \frac{|\nabla v_h|^2}{|\log (|x|/2)|^2} \, dx + 
        \int_{T\setminus B_h}  \frac{|\nabla v_h|^2}{|\log (|x|/2)|^2} \, dx
        \right]
        \\
        =&  
        \int_{B_h} \frac{|\nabla v_h|^2}{|\log (|x|/2)|^2} \, dx + 
        \sum _{T\in \mathcal{T}_h^2} \int_{T\setminus B_h}  \frac{|\nabla v_h|^2}{|\log (|x|/2)|^2} \, dx.
    \end{aligned}
\end{equation*}
The integrals over $T\setminus B_h$ can be estimated as before because $x\in T\setminus B_h$ implies $\abs{x}\ge h$, so
\begin{equation}\label{es:out:bh}
\sum _{T\in \mathcal{T}_h^2}\int_{T\setminus B_h}  \frac{|\nabla v_h|^2}{|\log (|x|/2)|^2} \, dx\ge \frac{1}{\abs{\log (h/2)}^2} \sum _{T\in \mathcal{T}_h^2} \int_{T\setminus B_h}  |\nabla v_h|^2 \, dx. 
\end{equation}
As far as the  integral over $B_h$,  we apply the coarea formula and integrate by parts to obtain
\begin{align*}
    \int_{B_h} \frac{|\nabla v_h|^2}{|\log (|x|/2)|^2} \, dx
    &= \int_0^h \frac{1}{|\log (s/2)|^2}\left(\int_{\partial B_s} |\nabla v_h|^2 \, d\mathcal{H}^2\right)\, ds\\
    &= \frac{1}{\abs{\log (h/2)}^2} \int_{B_h}  |\nabla v_h|^2 \, dx - 2\int_0^h \left(\int_{B_s}\abs{\nabla v_h}^2 \, dx\right) \frac{1}{\abs{\log(s/2)}^3}\, ds
    \\
    &\geq \frac{1}{\abs{\log (h/2)}^2} \int_{B_h}  |\nabla v_h|^2 \, dx - 2 \int_{B_h}\abs{\nabla v_h}^2 \, dx \, \int_0^h \frac{1}{\abs{\log(s/2)}^3} \, ds
    \\
    &= \frac{1}{\abs{\log (h/2)}^2} \left[1  - 2 \int_0^h \frac{\abs{\log (h/2)}^2}{\abs{\log(s/2)}^3} \, ds \right] \int_{B_h}  |\nabla v_h|^2 \, dx .
\end{align*}
Now, since
\begin{equation*}
    \lim_{h\to 0} \int_0^h \frac{\abs{\log (h/2)}^2}{\abs{\log(s/2)}^3} \, ds = 0
\end{equation*}
there exists a positive constant $C_2 < 1$ such that
\begin{equation}\label{es:bh}
    \int_{B_h} \frac{|\nabla v_h|^2}{|\log (|x|/2)|^2} \, dx\ge  \frac{C_2}{\abs{\log (h/2)}^2} \int_{B_h}  |\nabla v_h|^2 \, dx
\end{equation}
for all sufficiently small $h$. We can now sum up the estimates \cref{es:th1}, \cref{es:out:bh} and \cref{es:bh} to arrive at the claimed inequality \cref{claim-liviu}. The lower bound in \Cref{main2} is therefore proved.

\subsection{Proof of the upper bound}
\label{ss:ub-3d-general}
To prove the upper bound from \cref{main2}, it suffices to find a function $v_h \in V_h^3$ such that
\begin{equation}\label{e:ub-3d-strategy}
    \dfrac{\displaystyle\int_B \abs{\nabla v_h}^2 dx}{\displaystyle\int_B \abs{x}^{-2} \abs{v_h}^2 dx} \leq \dfrac{1}{4}+\dfrac{4\pi^2}{\abs{\log h}^2}+o\left(\dfrac{1}{\abs{\log h}^2}\right).
\end{equation}
As in \cref{ss:ub-1d}, we will take $v_h = \Pi_h v_\varepsilon$ to be the piecewise linear interpolation of a function $v_\varepsilon$ that is close to attaining the minimum in \cref{e:hardy-ratio-def-3d}, where $\varepsilon$ is a small parameter to be determined as a function of the mesh size $h$. We make here the particular choice
\begin{equation}\label{v_epsilon_n=3}
    v_\varepsilon(x)=
    \dfrac{1}{\sqrt{\abs{x}+\varepsilon}}\sin\left(\frac{\pi \log(\abs{x}+\varepsilon)}{\log \varepsilon}\right)
    - \dfrac{1}{\sqrt{1+\varepsilon}}\sin\left(\frac{\pi \log(1+\varepsilon)}{\log \varepsilon}\right).
\end{equation}
The following result follows from direct calculations.

\begin{lemma}\label[lemma]{estimates_n=3}
    For every $\varepsilon>0$, the function $v_\varepsilon$ defined in \cref{v_epsilon_n=3} belongs to $H^1_0(B) \cap H^2(B)$. In particular, for  $\varepsilon \ll 1$ we have
    \begin{equation}
      \begin{aligned}
            \frac{1}{\abs{B}}\int_B \abs{x}^{-2}\abs{v_\varepsilon(x)}^2dx&=\dfrac{\abs{\log \varepsilon}}{2}+O\left(\dfrac{\varepsilon}{\abs{\log \varepsilon}}\right)\\
            \frac{1}{\abs{B}}\int_B \abs{\nabla v_\varepsilon}^2 dx  & \leq \dfrac{\abs{\log \varepsilon}}{2}\left(\dfrac{1}{4}+\dfrac{\pi^2}{\abs{\log \varepsilon}^2}\right)+O\left(\varepsilon\right)\\
            \frac{1}{\abs{B}}\int_B\left(\abs{x}+\varepsilon\right)^2\abs{D^2 v_\varepsilon(x)}^2dx&\leq \dfrac{9}{32}\abs{\log \varepsilon}+O\left(\dfrac{\varepsilon}{\log \varepsilon}\right).
        \end{aligned}
   \end{equation}
\end{lemma}

We will also use the following estimates, which are similar to those in \cref{lem:norm-v-eps} and \cref{lem:norm-v-eps-3d}. The proof is similar to that of \cref{lem:norm-v-eps-3d}, which is reported in the appendix, except that one must use the finite element interpolation estimates from \cref{thint} instead of those in \cref{thint1}. The details are omitted for brevity.

\begin{lemma}\label[lemma]{first_lemma_dim3}
    For every mesh size $h$ and every function $f\in H^1_0(B)\cap H^2(B)$, set 
    \begin{equation*}\label{error_h}
      \mathcal{E}_h(f):= h^2 \norma{f}^2_{H^2(B)} +h \norma{f}_{H^2(B)}\norma{\nabla f}_{L^2(B)}. 
    \end{equation*}
    There exists a constant $C>0$, independent of both $f$ and $h$, such that 
    \begin{equation}
        \begin{aligned}
        \int_{B}\abs{\nabla \Pi_h f}^2 dx &\leq  \int_{B}\abs{\nabla f}^2 dx + C \mathcal{E}_h(f),\\
        \int_{B}\abs{x}^{-2} \abs{ \Pi_h f}^2 dx &\geq  \int_{B}\abs{x}^{-2} \abs{ f}^2 dx - C \mathcal{E}_h(f),
    \end{aligned}
    \end{equation}
\end{lemma}

With these results in hand, it is relatively straightforward to show that the function $v_h = \Pi_h v_\varepsilon \in V_h^3$ satisfies~\cref{e:ub-3d-strategy} for a suitable choice of $\varepsilon=\varepsilon(h)$. Indeed, by \Cref{estimates_n=3}, for every $\varepsilon>0$ we can estimate
\begin{equation*}
    \begin{aligned}
\norma{v_\varepsilon}^2_{H^2(B)}&\leq \int_B \abs{x}^{-2}\abs{v_\varepsilon}^2 dx+ \int_B \abs{\nabla v_\varepsilon}^2 dx +\varepsilon^{-2}\int_B \left(\abs{x}+\varepsilon\right)^2 \abs{D^2 v_\varepsilon}^2 dx \\
& \leq \dfrac{\abs{\log \varepsilon}}{\varepsilon^2}+O\left(\dfrac{1}{\varepsilon\abs{\log \varepsilon}}\right).
    \end{aligned}
\end{equation*}
This, in turn, implies that $\mathcal{E}_h(v_\varepsilon) \leq \left( h^2 \varepsilon^{-2}+ h \varepsilon^{-1} \right)\abs{\log \varepsilon}$.
Using \cref{estimates_n=3}, \cref{first_lemma_dim3}, and this last estimate we then find
\begin{align*}
        \dfrac{\displaystyle\int_B \abs{\nabla \Pi_h v_\varepsilon}^2 dx}{\displaystyle\int_B \abs{x}^{-2} \abs{\Pi_h v_\varepsilon}^2 dx}
        & \leq \dfrac{\displaystyle\int_B \abs{\nabla  v_\varepsilon}^2 dx+ C\mathcal{E}_h(v_\varepsilon)}{\displaystyle\int_B \abs{x}^{-2} \abs{ v_\varepsilon}^2 dx-C \mathcal{E}_h(v_\varepsilon)}\\
        & \leq \dfrac{\displaystyle\frac{1}{4}+\frac{\pi^2}{\abs{\log \varepsilon}^2}+O\left(\frac{\varepsilon}{\abs{\log \varepsilon}}\right)+C\left(\frac{h^2}{\varepsilon^2}+\frac{h}{\varepsilon}\right)}{\displaystyle 1+O\left(\frac{\varepsilon}{\abs{\log \varepsilon}^2}\right)-C\left(\frac{h^2}{\varepsilon^2}+\frac{h}{\varepsilon}\right) }.
\end{align*}
We now fix $\varepsilon=h \abs{\log h}^3$ and obtain 
\begin{equation*}
    \dfrac{\displaystyle\int_B \abs{\nabla \Pi_h v_\varepsilon}^2 dx}{\displaystyle\int_B \abs{x}^{-2} \abs{\Pi_h v_\varepsilon}^2 dx}
    \leq 
    \dfrac{\displaystyle\frac{1}{4}+
    \left( \frac{\pi}{\abs{\log h} - 3 \log\abs{\log h}} \right)^2+O\left( \frac{1}{\abs{\log h}^3} + h \abs{\log h}^2 \right)}{\displaystyle 1+O\left( \frac{1}{\abs{\log h}^3} + h \abs{\log h}^2 \right)}.
\end{equation*}
Since $\abs{\log h}^{-3} + h \abs{\log h}^2=o( \abs{\log h}^{-2})$, this inequality implies \cref{e:ub-3d-strategy} for all sufficiently small $h$ values. This concludes the proof of the upper bound part of \cref{main2}.

\section{Numerical results}\label{s:numerics}

Our analytical estimates in \Cref{main,,main-3d-spherical,,main2} provide precise asymptotic rates of convergence for the discrete Hardy constants in dimension $n=1$ and $n\geq 3$ as the mesh size $h$ tends to zero. This section reports some computational results to validate our estimates and check if they are quantitatively accurate for small but finite $h$, rather than just asymptotically.

\begin{figure}
    \centering
    \includegraphics[width=0.49\linewidth]{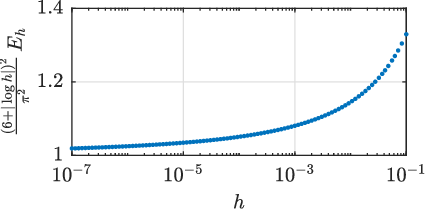}
    \hfill
    \includegraphics[width=0.49\linewidth]{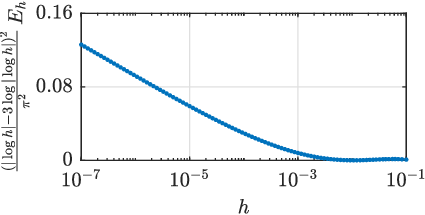}
    \vspace{-1ex}
    \caption{
        The gap $E_h = S_{h} - 1/4$ between the discrete and exact Hardy constant in dimension $n=1$, scaled by the factors $\pi^2/(6+\abs{\log h})^2$ (left) and $\pi^2/ ( \abs{\log h} - 3 \log \abs{\log h} )^2$ (right) predicted by the lower and upper bounds in \cref{e:lb-1d-statement,e:ub-1d-statement}, respectively.
        \label{err1}
    }
\end{figure}

\begin{figure}
    \centering
    \includegraphics[width=0.49\linewidth]{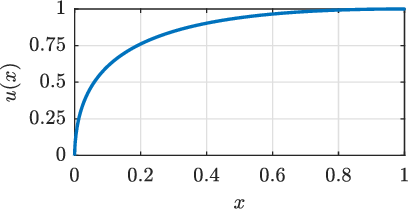}
    \hfill
    \includegraphics[width=0.49\linewidth]{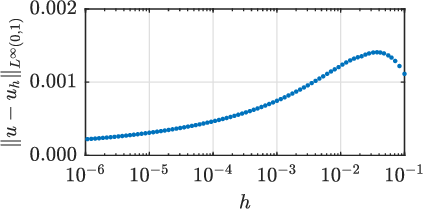}
    \vspace{-1ex}
    \caption{
        \emph{Left:} Minimizer $u \in V_h$ for~\cref{e:discrete-hardy-min} and $h=10^{-3}$, normalized so that $u(1)=1$.
        \emph{Right:} Error $\|u-u_h\|_{L^\infty(0,1)}$ between the minimizer $u$ of~\cref{e:discrete-hardy-min} and the function $u_h$ in~\cref{e:near-opt-u-1d}. Both functions are normalized so that $u(1)=u_h(1)=1$.
        \label{f:ev-1d}
    }
\end{figure}

\subsection{Computations for \texorpdfstring{$n=1$}{n=1}}
For the case of dimension $n=1$, the implementation of the finite element method is straightforward and the computation of the discrete Hardy constant $S_h$ amounts to solving a tridiagonal generalized eigenvalue problem. We solved this eigenvalue problem for uniform meshes with $N$ equispaced interpolation nodes, $x_k = k/N$ for $k=0,\ldots,N$. The mesh size is $h=1/N$. We considered 100 logarithmically spaced integer values $N$ from $N=10$ to $N=10^7$. The gap
\begin{equation*}
    E_h := S_{h} - \frac14
\end{equation*}
is plotted as a function of the mesh size in the two panels of \Cref{err1}, where it is compensated by the values $\pi^2/(6+\abs{\log h})^2$ and $\pi^2/ ( \abs{\log h} - 3 \log \abs{\log h} )^2$ that one predicts (up to higher-order corrections) from the lower and upper bounds in \cref{e:lb-1d-statement,e:ub-1d-statement}, respectively. We use these values instead of the simpler asymptotic predictions from \cref{main} because we expect them to be more precise for finite $h$ values. The left panel in \Cref{err1} suggests that the upper bound \cref{e:ub-1d-statement} on $S_{h}$ overestimates $E_h$. Note also how the $O(1/\abs{\log h}^2)$ asymptotic behavior of $E_h$, guaranteed by \cref{main}, is not evident in the plot despite the very small mesh sizes. This is due to the extremely slow decay of other, higher-order logarithmic corrections. In contrast, the lower bound from \cref{e:lb-1d-statement} predicts $E_h$ much more accurately for the mesh sizes $h$ in our numerical computations, even though it was obtained by replacing the finite element space $V_h$ with the strictly larger space $U_h$. 

The accuracy of our lower bound analysis is further confirmed if we consider the minimizer $u \in V_h$ for~\cref{e:discrete-hardy-min}. This is the principal eigenfunction of the eigenvalue problem for $S_h$ and is shown in the left panel of \Cref{f:ev-1d} for $h=10^{-3}$ (results for other values of $h$ are similar). As shown by the right panel in the same figure, this eigenfunction is approximated well by the function $u_h$ in~\cref{e:near-opt-u-1d} in a pointwise sense. Moreover, the approximation appears to improve as the mesh size $h$ decreases. This suggests that a careful interpolation of $u_h$ may improve the upper bound in \cref{e:ub-1d-statement} so that it matches more precisely the lower bound in~\cref{e:lb-1d-statement}.

\begin{figure}[t]
    \centering
    \includegraphics[width=0.49\linewidth]{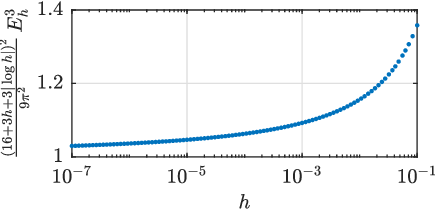}
    \hfill
    \includegraphics[width=0.49\linewidth]{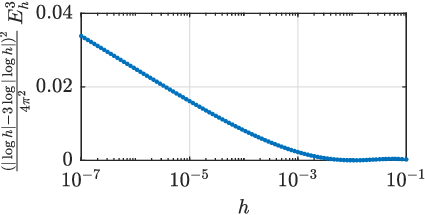}
    \caption{
        The gap $E_h^3 = S_{h}^3 - 1/4$ between the discrete Hardy constant for $n=3$ dimensions, computed with \cref{Snh}, and the exact value $S^3 = 1/4$. Results are plotted after scaling by the functions $9\pi^2/(16 + 3h + 3\abs{\log h})^2$ (right) and $4\pi^2/ ( \abs{\log h} - 3 \log \abs{\log h} )^2$ (right) predicted by the lower and upper bounds in \cref{e:lb-3d-statement,e:ub-3d-statement}, respectively.
        \label{err2}
    }
\end{figure}

\begin{figure}[t]
    \centering
    \includegraphics[width=0.49\linewidth]{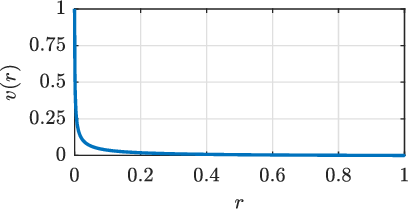}
    \hfill
    \includegraphics[width=0.49\linewidth]{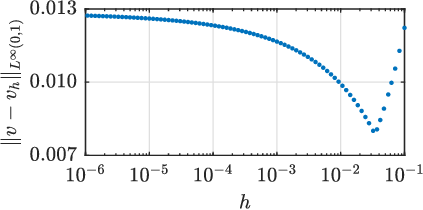}
    \vspace{-1ex}
    \caption{
        \emph{Left:} Minimizer $v \in V_h$ for~\cref{Snh} with $n=3$ and $h=10^{-2}$, normalized so that $v(0)=1$.
        \emph{Right:} Error $\|v-v_h\|_{L^\infty(0,1)}$ between the minimizer $v$ of~\cref{Snh} and the function $v_h$ in~\cref{e:optimal-function-3d} for $n=3$. Both functions are normalized so that $v(0)=v_h(0)=1$.
        \label{f:ev-3d}
    }
\end{figure}

\subsection{Computations for \texorpdfstring{$n=3$}{n=3} with radial symmetry}

Next, we consider computations in dimension $n=3$ when the domain $\Omega=B$ is the unit ball. We focus on the case of radially symmetric meshes because the computation of $S_h^3$ reduces to solving the one-dimensional minimization problem in \cref{Snh} for $n=3$. This problem is equivalent to a tridiagonal generalized eigenvalue problem, which can be solved on a laptop even for very small values of the mesh size $h$. General meshes of the three-dimensional ball, instead, would require a more sophisticated parallel implementation on a computer cluster that is beyond the scope of our work. Our computations used uniform meshes with $N$ elements of size $h=1/N$, and we considered 100 logarithmically spaced integer values $N$ from $N=10$ to $N=10^7$. The gap
\begin{equation*}
    E_h^3 := S_{h}^3- \frac14
\end{equation*}
is plotted in \Cref{err2} after scaling by the values $9\pi^2/(16 + 3h + 3 \vert\log h\vert)^2$ and $4\pi^2/ ( \abs{\log h} - 3 \log \abs{\log h} )^2$ one predicts for $E_h^3$ (up to higher-order corrections) from the lower and upper bounds on $S_h^3$ in \cref{e:lb-3d-statement} and \cref{e:ub-3d-statement}, respectively. As before, we use these values rather than the asymptotically equivalent predictions from \Cref{main-3d-spherical} because we expect them to be more accurate for small but finite $h$. Again, the results suggest that our lower bound predicts $E_h^3$ more accurately than our upper bound. Moreover, \Cref{f:ev-3d} reveals that the function $v_h$ in \cref{e:optimal-function-3d} is a reasonable approximation for the minimizer $v$ of \cref{Snh}. However, contrary to the case of  dimension $n=1$, the approximation error does not seem to decrease as the mesh is refined.

\section{Open problems}\label{s:discussion}
We conclude with a list of open problems. 
\begin{enumerate}[leftmargin=*, align=left, itemsep=0.5ex]
    \item Determine the exact prefactor for the $O(1/\abs{\log(h)}^2)$ correction to the asymptotic value of the discrete Hardy constant in $n\geq 3$ dimensions. 
    Our analysis does not provide this because the leading-order corrections in the upper and lower bounds in \Cref{main-3d-spherical} do not match as the mesh size $h$ tends to zero. The good quantitative agreement between the values of $S_h^n$ computed numerically in \cref{s:numerics} and the lower bounds in \cref{e:lb-3d-statement} suggests that
    \begin{equation*}
        S_h^n = S^n + \frac{\pi^2}{\abs{\log(h)}^2}+o\left(\dfrac{1}{\abs{\log(h)}^2}\right).
    \end{equation*}
    Confirming this precisely, even with computer assistance, is challenging due to the slow decay of higher-order logarithmic corrections. If this prediction is correct, however, then one should be able to improve upper bound in \cref{e:ub-3d-statement}. 
          
    \item Generalize our estimates to the Hardy inequality with exponent $p\in(1,n)$. The calibration technique we employed has already been used to prove sharp lower bounds on the optimal constant for the one-dimensional Poincar\'e inequality in $W^{1,p}_0$ when $p$ is an even integer \cite{Chernyavsky2021}. We wonder if those arguments carry over first to the Hardy inequality with exponent $p\neq 2$, and then to its finite element approximations.

    \item Generalize our estimates to refinements of the Hardy inequality in dimension $n=2$, which was not considered here. In particular, we wonder if one can estimate the convergence rate of finite element approximations for logarithmic versions of the Hardy inequality (see, e.g., \cite{pino}). 
\end{enumerate}

\appendix
\section{Proof of \texorpdfstring{\Cref{main-3d-spherical}}{Theorem \ref{main-3d-spherical}}}
\label{s:proof-3d-radial}

In this appendix, we prove the upper and lower bounds reported in \Cref{main-3d-spherical}, which apply to the discrete Hardy constant $S_h^n$ defined in \cref{e:hardy-ratio-def-3d} for any dimension $n\geq 3$. We follow essentially the same strategy used in \cref{s:1} for the case of $n=1$ dimensions, but the computations are more involved. There is only one minor technical difference in the proof of the lower bound, which we point out explicitly.

\subsection{Proof of the lower bound}
\label{ss:lb-3d}
Let us first prove the lower bound part of \Cref{main-3d-spherical}. We shall in fact establish the lower bound
\begin{equation}\label{e:lb-3d-statement}
    S_h^n \geq S^n + \frac{\pi^2}{\left( \frac{8(n-1)}{n(n-2)} + h + \abs{ \log h }\right)^2}  +  o\left( \frac{1}{\abs{ \log h }^2} \right),
\end{equation}
which is asymptotically equivalent to that in the theorem for $h\ll 1$. The proof follows the same strategy as in \cref{ss:lb-1d} with only one difference: the space $U_h$ in that section is replaced by the space $W_h$ of functions in $H^1(0,1)$ that vanish at $r=1$ and that are linear both on $(0,h)$ and on $(1-h,1)$. Since the space $W_h$ contains the finite element space $V_h$ used to define $S_h^n$, the inequality in \cref{e:lb-3d-statement} follows from a lower bound on
\begin{equation}\label{e:mu-def-3d}
    \mu_{h}^n := \inf_{v \in W_h} \dfrac{\displaystyle{\int_{0}^{1} r^{n-1} \abs{v'}^{2}\,dr}}{\displaystyle{\int_{0}^{1} r^{n-3} v^{2}\,dr}}.
\end{equation}

We will compute this quantity using a calibration-type argument. To simplify the notation, let us introduce for every positive integer $m$ the function
\begin{subequations}\label{e:fg-def}
    \begin{equation}\label{e:f-def}
        f_m(h) := \frac1m \sum_{k=0}^{m-1} (1-h)^{2-n+k}.
    \end{equation}
    Let us also set
    \begin{equation}\label{e:g-def}
        g_n(h) := f_n(h) - 2 f_{n-1}(h) +  f_{n-2}(h).
    \end{equation}
\end{subequations}

\begin{theorem}\label{th:mu-3d}
    Suppose $\delta_h$ and $\gamma_h$ solve
    \begin{subequations}
    \label{e:eps-equation-3d}
        \begin{gather}
        \label{e:eps-cond-3d-a}
            \left( \frac{n^2-2n}{2} - 2\delta_h^2 \right)\left( \frac{n^2}{4} - \delta_h^2 + n \delta_h \tan\left( \gamma_h + \delta_h \log h \right)\right) = \frac{n-2}{n-1} \left( \frac{n^2}{4} +\delta_h^2 \right)^2
            \\[1ex]
            \label{e:eps-cond-3d-b}
            \frac{n-2}{2} + \delta_h \tan\left( \gamma_h + \delta_h \log (1-h) \right)  = 
            \frac{1}{h}
            \left[ f_n(h) - g_n(h)\left( \frac{(n-2)^2}{4} + \delta_h^2 \right)\right].
        \end{gather}
    \end{subequations}
    Then, $\mu_{h}^n = S^n + \delta_h^2$. In particular, for $h\ll 1$ we have
    \begin{equation}\label{e:mu-expansion-3d}
         \mu_{h}^n = S^n + \frac{\pi^2}{\left( \frac{8(n-1)}{n(n-2)} + h + \abs{ \log h }\right)^2}  +  o\left( \frac{1}{\abs{ \log h }^2} \right).
    \end{equation}
\end{theorem}

This result is an immediate consequence of the following three lemmas. The first establishes the lower bound $\mu_{h}^n\geq S^n + \delta_h^2$. The second provides an asymptotic expression for $\delta_h$ when $h\ll1$. The third shows that $\mu_{h}^n\leq S^n + \delta_h^2$, from which we conclude that $\mu_{h}^n = S^n + \delta_h^2$.

\begin{lemma}\label[lemma]{lem:mu-lb-3d}
    If $\delta_h$ and $\gamma_h$ solve \cref{e:eps-equation-3d}, then $\mu_h \geq S^n + \delta_h^2$.
\end{lemma}

\begin{proof}
    Set
    \begin{equation*}
        F_\lambda(v) := \int_{0}^{1} r^{n-1}\abs{v'}^{2} - \lambda r^{n-3} v^2\,dr
    \end{equation*}
    and observe that
    \begin{equation}\label{e:mu-ineq-form-3d}
        \mu_{h}^n = \max\left\{ \lambda:\; F_\lambda(v) \geq 0 \quad \forall v \in W_h\right\}.
    \end{equation}
    By definition of $W_h$, every $v \in W_h$ satisfies
    \begin{equation*}
        v(x) =  
        \begin{cases}
            h^{-1}(h-r) \,v(0) + h^{-1}r \,v(h) &\text{if } r \in (0,h),\\
            h^{-1}(1-r) \,v(1-h)&\text{if } r \in (1-h,1).
        \end{cases}
    \end{equation*}
    Moreover, every continuously differentiable function $\varphi(r)$ on $[h,1-h]$ satisfies
    \begin{equation*}
    \int_h^{1-h} \left(r^{n-2} \varphi(r) \, v^2  \right)' \,dr 
    = (1-h)^{n-2}\varphi(1-h) v(1-h)^2  - h^{n-2} \varphi(h) v(h)^2
    \end{equation*}
    by the fundamental theorem of calculus.
    Expanding the derivative under the integral using the chain rule, and using the piecewise linear representation of $v$, we can rewrite
    \begin{multline}\label{e:F-expansion-3d}
        F_\lambda(v) = 
        \frac{h^{n-2}}{n(n-1)(n-2)}\begin{pmatrix} v(0)\\ v(h)\end{pmatrix}^\top 
        M(n,h,\lambda)
        \begin{pmatrix} v(0)\\ v(h)\end{pmatrix}
        \\[1ex]
        + \int_h^1  r^{n-1} \vert v' \vert^{2} + 2 r^{n-2} \varphi(r) v v' +\left[ (r^{n-2}\varphi)' - \lambda \right] v^2 \;dx 
        \\[1ex]
        + h^{n-3}
        \left[ f_n(h)  - g_n(h)\lambda  - h \varphi(1-h)\right] v(1-h)^2,
    \end{multline}
    where the functions $f_n$ and $g_n$ are as in \cref{e:fg-def} and
    \begin{equation*}
    M(n, h,\lambda) :=
        \begin{pmatrix} 
            (n-1)(n-2)-2\lambda &  (n-2)(1-n-\lambda)\\
            (n-2)(1-n-\lambda) & (n-1)(n-2)\left(1-\lambda + n\varphi(h)\right)
        \end{pmatrix}.
    \end{equation*}
    The right-hand side of \cref{e:F-expansion-3d} is nonnegative if the matrix $M(n,h,\lambda)$ is positive semidefinite, the integrand in the second line is a square, and the coefficient of the last term is nonnegative. This is true if $\varphi$ and $\lambda$ satisfy
    \begin{subequations}\label{e:phi-conditions-3d}
         \begin{gather}
         (n-1)(n-2)-2\lambda \geq 0,
         \label{e:phi-conditions-3d-a}
         \\
         (n-1)\left[ (n-1)(n-2)-2\lambda \right]\left[ 1-\lambda + n\varphi(h)\right]
         \geq (n-2)(1-n-\lambda)^2,
         \label{e:phi-ineq-det-3d}
         \\
         {f_n(h)}  - g_n(h)\lambda  -  h \varphi(1-h)\geq 0, \label{e:phi-ineq-h-3d}
         \\
         \left( r^{n-2}\varphi \right)' - \lambda = r^{n-3} \varphi^2. \label{e:phi-ode-3d}
        \end{gather}
    \end{subequations}
    
    To satisfy these four conditions, we set
    $\lambda = S^n + \delta^2$ for some $\delta$ to be determined below, and solve the differential equation in \cref{e:phi-ode-3d} to find
    \begin{equation*}
        \varphi(r) = \frac{n-2}{2} + \delta \tan \left( \gamma + \delta \log r \right),
    \end{equation*}
    where $\gamma$ is an integration constant. Then, we substitute this function into \cref{e:phi-conditions-3d-a,,e:phi-ineq-det-3d,,e:phi-ineq-h-3d}. After some rearrangement, we conclude that $\delta$ and $\gamma$ should satisfy
    \begin{subequations}
        \begin{gather}\label{e:eps-choice-condition-3d}
        n(n-2) - 4 \delta^2 \geq 0,
        \\
        \left( \frac{n^2-2n}{2} - 2\delta^2 \right)\left( \frac{n^2}{4} - \delta^2 + n \delta \tan\left( \gamma + \delta \log h \right)\right) 
        \geq
        \frac{n-2}{n-1} \left( \frac{n^2}{4} +\delta^2 \right)^2,
        \\
        \frac{n-2}{2} + \delta \tan\left( \gamma + \delta \log (1-h) \right)  \leq 
        \frac{1}{h}
        \left[ f_n(h)  - g_n(h)\frac{(n-2)^2}{4} - g_n(h)\delta^2\right].
    \end{gather}
    \end{subequations}
    These inequalities hold with equality if $\delta = \delta_h$ and $\gamma=\gamma_h$. The choice $\lambda=S^n + \delta_h^2$ is thus feasible for the maximization problem in \cref{e:mu-ineq-form-3d}, so $\mu_h \geq S^n + \delta_h^2$.
\end{proof}

Next, we solve~\cref{e:eps-equation-3d} for $h\ll 1$ to derive an asymptotic expansion for $\delta_h$. 

\begin{figure}
    \centering
    \includegraphics[width=0.49\linewidth]{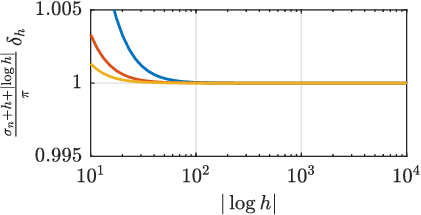}
    \hfill
    \includegraphics[width=0.49\linewidth]{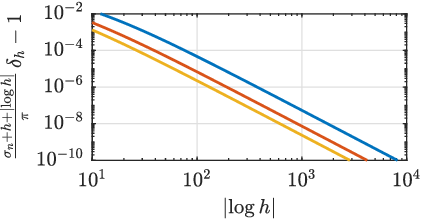}
    \caption{\emph{Left:} Ratio of $\delta_h$ to the leading-order term $\pi/(\sigma_n + h +\abs{ \log h })$ in its asymptotic expansion, where $\sigma_n = 8(n-1)/(n^2-2n)$. \emph{Right:} Error between this ratio and the value $1$. The three curves are for $n=3$ (blue), $n=4$ (red), and $n=5$ (yellow).
    }
    \label{fig:asympt-3d}
\end{figure}

\begin{lemma}
    If $\delta_h$ and $\gamma_h$ solve \cref{e:eps-equation-3d} and $h \ll 1$, then
    \begin{equation}\label{e:delta-3d-asympt}
        \delta_h = \frac{\pi}{ \frac{8(n-1)}{n(n-2)} + h + \abs{\log h}} + o\left( \frac{1}{\abs{ \log h }}\right).
    \end{equation}
\end{lemma}

\begin{proof}
    We perform an asymptotic solution of~\cref{e:eps-equation-3d} for $h\ll 1$. Anticipating that $\delta_h \ll 1$ and that $\gamma_h \approx \pi/2$ in this regime, we rearrange \cref{e:eps-cond-3d-a} keeping only the leading-order terms to find that
    \begin{equation*}
         \frac{n^2}{4} + n \delta_h \tan\left( \gamma_h + \delta_h \log h\right) = \frac{n^3}{8(n-1)}.
    \end{equation*}
    This equation can be solved for $\gamma_h$ to obtain, again to leading order,
    \begin{equation*}
        \gamma_h = - \tan^{-1}\left( \frac{n(n-2)}{8(n-1)\delta_h} \right) - \delta_h \log h.
    \end{equation*}
    We then rearrange \cref{e:eps-cond-3d-b} keeping only the leading-order terms in $\delta_h$ to arrive, after some algebraic simplifications, at
    \begin{equation*}
         \gamma_h = 
            \tan^{-1}\left(  \frac{1}{h\delta_h } \right).
    \end{equation*}
    Substituting the expression for $\gamma_h$ and using the Taylor expansion of the tangent for small $\delta_h$ we obtain, to leading order in $\delta_h$,
    \begin{equation*}
         - \frac{\pi}{2}  + \frac{8(n-1)}{n(n-2)} \delta_h  + o(\delta_h) - \delta_h \log h = 
           \frac{\pi}{2} - h\delta_h + o(\delta_h).
    \end{equation*}
    Solving this equation for $\delta_h$ yields \cref{e:delta-3d-asympt}. The correctness of this asymptotic expansion confirmed for $n=3$, $4$ and $5$ by \Cref{fig:asympt-3d}.
\end{proof}

Finally, we show that the lower bound $\mu_{h}^n\geq S^n + \delta_h^2$ proved in \cref{lem:mu-lb-3d} is sharp by establishing the reverse inequality.

\begin{lemma}
     If $\delta_h$ and $\gamma_h$ solve \cref{e:eps-equation-3d}, then $\mu_h \geq S^n + \delta_h^2$.
\end{lemma}
\begin{proof}
    Let $\lambda=S^n+\delta_h^2$. It suffices to find a function $v_h \in W_h$ for which $F_\lambda(v_h)=0$, because then
    \begin{equation*}
        \lambda = S^n+\delta_h^2 = \dfrac{\displaystyle{\int_{0}^{1} r^{n-1} \abs{v'}^{2}\,dr}}{\displaystyle{\int_{0}^{1} r^{n-3} v^{2}\,dr}} \geq \mu_h^n.
    \end{equation*} 
    
    To construct $v_h$, recall that $\delta_h$ and $\gamma_h$ are chosen to satisfy \cref{e:phi-ineq-h-3d,e:phi-ineq-det-3d} with equality. Thus, for any $v \in W_h$, identity \cref{e:F-expansion-3d} becomes
    \begin{align*}
        F_\lambda(v)
        = 
        &\int_h^{1-h}  r^{n-3} \left( r v' + \varphi v \right)^2 \;dr
        \\
        &+\frac{h^{n-2}}{n(n-1)}\begin{pmatrix} v(0)\\ v(h)\end{pmatrix}^\top 
        \begin{pmatrix} 
            n-1-\frac{2\lambda}{n-2} &  1-n-\lambda\\
            1-n-\lambda & \frac{(n-2)(1-n-\lambda)^2}{(n-1)(n-2)-2\lambda}
        \end{pmatrix}
        \begin{pmatrix} v(0)\\ v(h)\end{pmatrix} .
    \end{align*}
    Recognizing that the $2\times 2$ matrix in the second line has rank one, we can further rewrite
    \begin{align*}
        F_\lambda(v)
        =
        &\int_h^{1-h}  r^{n-3} \left( r v' + \varphi v \right)^2 \;dr
        \\ 
        &+ \frac{h^{n-2}}{n(n-1)}
        \left[ \sqrt{n-1-\frac{2\lambda}{n-2}} \, v(0) + \frac{\sqrt{n-2}(1-n-\lambda)}{\sqrt{(n-1)(n-2)-2\lambda}} \, v(h)\right]^2 .
    \end{align*}
    To construct $v_h$, therefore, we first solve the differential equation $r v' + \varphi(r) v = 0$ on $(h,1-h)$. Then, we extend the solution linearly to $(0,h)$ and to $(1-h,1)$ while satisfying the boundary conditions
    \begin{equation*}
        v(1)=0
        \qquad\text{and}\qquad
       v(0) = - \frac{(n-2)(1-n-\lambda)}{(n-1)(n-2)-2\lambda} \, v(h).
    \end{equation*}
    Introducing the function
    \begin{equation*}
        \Psi(r) := r^{-\frac{n-2}{2}} \cos\left( \gamma_h + \delta_h \log r \right)
    \end{equation*}
    for convenience, we find that
    \begin{equation}\label{e:optimal-function-3d}
        v_h(r) =
        \begin{cases}
            A h^{-1}\left[ r  - \frac{(n-2)(1-n-\lambda)}{(n-1)(n-2)-2\lambda} (h-r)\right]\Psi(h) &r\in[0,h], \\[1ex]
            A \Psi(r) &r\in[h,1-h], \\[1ex]
            A h^{-1} (1-r) \Psi(1-h) &r\in[1-h,1],
        \end{cases}
    \end{equation}
    where $A$ is an arbitrary normalization constant.
\end{proof}

\subsection{Proof of the upper bound}
The upper bound in \cref{main-3d-spherical} is proven for any $n\geq 3$ exactly like its counterpart for $n=1$ in \cref{main}. The only difference is that we replace the function $v_\varepsilon$ in \cref{min1} with
\begin{equation}\label{e:v-interp-n}
    v_\varepsilon(x)=
    \begin{cases}
        0,
        & r\in [0,\varepsilon],\\
        r^{-\frac{n-2}{2}}\sin\left(\dfrac{\pi \log r}{\log \varepsilon}\right),
        & r\in (\varepsilon, 1],
    \end{cases}
\end{equation}
where $\varepsilon=\varepsilon(h)$ will be chosen to be an interpolation node. Observe that $v_{\varepsilon}(1)=0$ and $v_\varepsilon\in H^1(0,1)\cap H^2(\varepsilon,1)$. Moreover, direct calculation gives the following results.

\begin{lemma}\label[lemma]{lem:norm-v-eps-3d}
    For every $\varepsilon<1$, the function $v_\varepsilon$ in \cref{e:v-interp-n} satisfies
    \begin{gather*}
    \int_0^1 \abs{v_\varepsilon}^2 r^{n-3}dr=\dfrac{\abs{\log\varepsilon}}{2},
    \\
    \int_0^1 \abs{v'_\varepsilon}^2 r^{n-1}dr= \dfrac{(n-2)^2}{8}\abs{\log\varepsilon}+\dfrac{\pi^2}{2\abs{\log\varepsilon}},\\
    \int_\varepsilon^1 \abs{v''_\varepsilon}^2 r^{n+1}dr=\dfrac{n^2(n-2)^2}{32}\abs{\log \varepsilon}+\dfrac{\pi^2(n^2-2n+2)}{4\abs{\log \varepsilon}}+\dfrac{\pi^4}{2\abs{\log \varepsilon}^3}.
    \end{gather*}
\end{lemma}

Next, we derive two useful estimates that extend to $n\geq 3$ those stated for $n=1$ in \cref{lem:1d-estimates}. Note that the same arguments given here apply when $n=1$, too, and that in that case the assumption $f(1)=0$ can be dropped.

\begin{lemma}\label[lemma]{lem:estimates}
    Fix an integer $n\geq 3$ and let $\varepsilon=mh\in(0,1)$ be an interpolation node. Let $f\in H^1(0,1)\cap H^2(\varepsilon,1)$
    vanish on $[0, \varepsilon]$ and satisfy $f(1) = 0$. Set 
    \begin{equation*}
        \mathcal{E}_h(f):=\frac{h}{\varepsilon^{\frac{n+1}{2}}}
        \left(\int_0^1 \abs{f'}^2 r^{n-1}dr\right)^{\frac{1}{2}}    \left(\int_\varepsilon^1 \abs{f''}^2 r^{n+1}dr\right)^{\frac{1}{2}}+\frac{h^2}{\varepsilon^{n+1}}\int_\varepsilon^1\abs{f''}^2 r^{n+1}dr.
    \end{equation*}
    There exists a constant $C > 0$, independent of $f$, $h$ and $\varepsilon$, such that
    \begin{subequations}
        \begin{align}
        \label{first_e}
        \int_0^1 \abs{\left(\Pi_h f\right)'}^2 r^{n-1} dr 
        &\leq \int_0^1\abs{f'}^2r^{n-1}dr+C \mathcal{E}_h(f),
        \\
        \label{sec_e}   
        \int_0^1 \abs{\Pi_h f}^2 r^{n-3} dr 
        &\geq \int_0^1\abs{f}^2r^{n-3}dr-C \mathcal{E}_h(f).
    \end{align}
    \end{subequations}
\end{lemma}
\begin{proof}
   We first prove \cref{first_e}. Since $\varepsilon$ is an interpolation node, we have $f(r) = \Pi_h f(r) = 0$ for every $r \in [0,\varepsilon]$. We can therefore estimate
    \begin{align}
        \int_0^1\abs{\left(\Pi_h f\right)'}^2 r^{n-1} dr 
        =&\int_\varepsilon^1\abs{f'+\left(\Pi_h f-f\right)'}^2 r^{n-1}dr
        \nonumber \\
        \leq& \int_\varepsilon^1 \abs{f'}^2r^{n-1}dr +  \int_\varepsilon^1\abs{\left(\Pi_h f-f\right)'}^2 r^{n-1}dr 
        \nonumber \\
        &+2\left(\int_\varepsilon^1 \abs{f'}^2r^{n-1}dr \right)^{\frac{1}{2}}\left(\int_\varepsilon^1\abs{\left(\Pi_h f-f\right)'}^2 r^{n-1}dr\right)^{\frac{1}{2}}.
        \label{app_f}
    \end{align}
    We now use the bounds $\varepsilon\leq r\leq 1$ and the interpolation inequality \cref{gradiente_pol} on the interval $(\varepsilon,1)$, which is valid for every $f\in H^2(\varepsilon,1)$, to further estimate \cref{app_f} as 
    \begin{equation*}
        \begin{aligned}
                \int_0^1\abs{\left(\Pi_h f\right)'}^2 r^{n-1} dr
                \leq&  \int_\varepsilon^1 \abs{f'}^2 r^{n-1}dr+c^2h^2\int_\varepsilon^1\abs{f''}^2dr\\
                &
                +2ch \left(\int_\varepsilon^1 \abs{f'}^2r^{n-1}dr \right)^{\frac{1}{2}}\left(\int_\varepsilon^1\abs{f''}^2 dr\right)^{\frac{1}{2}}\\
                \leq& \int_\varepsilon^1 \abs{f'}^2 r^{n-1}dr+c^2h^2\varepsilon^{-(n+1)}\int_\varepsilon^1\abs{f''}^2 r^{n+1}dr\\
                &+2ch \varepsilon^{-\frac{n+1}{2}}\left(\int_\varepsilon^1 \abs{f'}^2r^{n-1}dr \right)^{\frac{1}{2}}\left(\int_\varepsilon^1\abs{f''}^2 r^{n+1} dr\right)^{\frac{1}{2}}.
        \end{aligned}
    \end{equation*}
    This implies \cref{first_e} for any constant $C \geq C_1 :=\max\{2c, c^2\}$. This constant can be taken to be independent of $\varepsilon$ because the constant $c$ in the interpolation inequality \cref{gradiente_pol} is an increasing function of the diameter of the integration domain \cite{raviart}, and can therefore be replaced by a larger constant (also denoted by $c$) independently of $\varepsilon$.

    Next, we derive inequality \cref{sec_e} using similar arguments. We start from the estimate
    \begin{align}
        \int_0^1 \abs{\Pi_h f}^2 r^{n-3} dr
        \geq& 
        \int_0^1 \abs{f}^2r^{n-3}dr-\int_0^1\abs{\Pi_h f-f}^2 r^{n-3} dr
        \nonumber \\
        &-2\left(\int_0^1 \abs{f}^2r^{n-3}dr \right)^{\frac{1}{2}}\left(\int_0^1\abs{\Pi_h f-f}^2 r^{n-3} dr\right)^{\frac{1}{2}}.
        \label{app_2}
    \end{align}
    Since $f(1)=0$ by assumption, then both $f$ and $\Pi_h f-f$ satisfy the Hardy inequality with optimal constant $S^n$. Using this inequality, and the fact that $\Pi_h f=f$ in $(0,\varepsilon)$ by assumption, we can further estimate \cref{app_2} as
    \begin{equation*}
        \begin{aligned}
             \int_0^1 \abs{\Pi_h f}^2 r^{n-3} dr\geq& 
             \int_0^1 \abs{f}^2r^{n-3}dr
             -S^n\int_\varepsilon^1\abs{\left(\Pi_h f-f\right)'}^2 r^{n-1} dr\\
             &-2S^n\left(\int_0^1 \abs{f'}^2r^{n-1}dr \right)^{\frac{1}{2}}\left(\int_\varepsilon^1\abs{\left(\Pi_h f-f\right)'}^2 r^{n-1} dr\right)^{\frac{1}{2}}.
        \end{aligned}
    \end{equation*}
    Applying to this inequality the same estimates used on \cref{app_f} yields \cref{sec_e} for any constant $C \geq C_2 :=\max\{2S^nc,S^n c^2\}$, which may be again chosen to be independent of $\varepsilon$. Estimates \cref{first_e,sec_e} clearly hold simultaneously for $C= \max\{C_1,C_2\}$.
\end{proof}

Finally, we establish the following refinement of the upper bound on $S_h^n$ from \cref{main-3d-spherical}, which is asymptotically equivalent to the latter when $h \ll 1$.

\begin{theorem}
    For every $n\geq 3$ and all sufficiently small mesh sizes $h$,
    \begin{equation}\label{e:ub-3d-statement}
        S_h^n \leq \dfrac{(n-2)^2}{4} + 
            \left( \dfrac{ (n+1) \pi  }{  2\abs{\log h} -6 \log\abs{\log h} } \right)^2
            + o\left( \frac{1}{ \abs{ \log h }^2 }\right).
    \end{equation}
\end{theorem}

\begin{proof}
    Let $\varepsilon = mh$ be an interpolation node.
    We can then follow exactly the same steps as in \cref{ss:ub-1d}, except that we let $v_\varepsilon$ be defined as in \cref{e:v-interp-n} and replace the results in \cref{lem:norm-v-eps,lem:1d-estimates} with those in \cref{lem:norm-v-eps-3d,lem:estimates}. We obtain
    \begin{equation*}
        S^n_h\leq \dfrac{\dfrac{(n-2)^2}{4}+\dfrac{\pi^2}{\abs{\log \varepsilon}^2}+C\left(h \varepsilon^{-\frac{n+1}{2}}+h^2\varepsilon^{-(n+1)}\right)}{1-C\left(h \varepsilon^{-\frac{n+1}{2}}+h^2\varepsilon^{-(n+1)}\right)}. 
    \end{equation*}
    for some constant $C$ independent of both $h$ and $\varepsilon$. 
    We now recall that $\varepsilon = mh$ and set 
    $$m= \left\lfloor h^{\frac{1-n}{n+1}} \, \abs{ \log h }^{\frac{6}{n+1}} \right\rfloor.$$ This gives $\varepsilon\sim h^{2/(n+1)} \abs{ \log h }^{6/(n+1)}$ and, in particular, $\varepsilon \leq h^{2/(n+1)} \abs{ \log h }^{6/(n+1)}$. We can thus find another constant, also denoted by $C$ and independent of $h$, such that
    \begin{equation*}
        S^n_h
        \leq 
        \dfrac{
            \dfrac{(n-2)^2}{4} + 
            \left( \dfrac{ (n+1) \pi  }{  2\abs{\log h} -6 \log\abs{\log h} } \right)^2
            +C\left( \dfrac{1}{\abs{ \log h}^3} + \dfrac{1}{\abs{ \log h}^6}\right)
        }
        {1-C \left( \dfrac{1}{\abs{ \log h}^3} + \dfrac{1}{\abs{ \log h}^6}\right) }.
    \end{equation*}
    This inequality implies \cref{e:ub-3d-statement} when the mesh size $h$ is sufficiently small.
\end{proof}
     
\section*{Acknowledgements and Declarations}
F. Della Pietra was supported by the MIUR-PRIN 2017 grant ``Qualitative and quantitative aspects of nonlinear PDEs'', and FRA Project (Compagnia di San Paolo and Universit\`a degli studi di Napoli Federico II) \verb|000022--ALTRI_CDA_75_2021_FRA_PASSARELLI|.

F. Della Pietra, A.L. Masiello and G. Paoli were supported by Gruppo Nazionale per l’Analisi Matematica, la Probabilità e le loro Applicazioni (GNAMPA) of Istituto Nazionale di Alta Matematica (INdAM).

L. I. Ignat was supported by project PN-III-P1-1.1-TE-2021-1539 of Romanian Ministry of Research, Innovation and Digitization, CNCS-UEFISCDI,  within PNCDI III.

G. Paoli was supported by the Alexander von Humboldt Foundation through an Alexander von Humboldt research fellowship. 

E. Zuazua has been funded by the Alexander von Humboldt-Professorship program, the ModConFlex Marie Curie Action, HORIZON-MSCA-2021-DN-01, the COST Action MAT-DYN-NET, the Transregion 154 Project “Mathematical Modelling, Simulation and Optimization Using the Example of Gas Networks” of the DFG, grants PID2020-112617GB-C22 and TED2021-131390B-I00 of MINECO (Spain), and by the Madrid Goverment – UAM Agreement for the Excellence of the University Research Staff in the context of the V PRICIT (Regional Programme of Research and Technological Innovation).

\bibliographystyle{plain}
\bibliography{biblio}

\end{document}